\documentclass[9pt,twocolumn]{IEEEtran}
\usepackage{amsmath,amssymb,mathrsfs,euscript,yfonts,psfrag,latexsym,dsfont,graphicx,bbm,color,amstext,wasysym,subfig,flushend,parskip,url,soul,bm,cite,balance}
\usepackage{amsthm}
\usepackage{algorithm}
\usepackage{algpseudocode}
\usepackage{algorithmicx}
\usepackage[symbol]{footmisc}
\usepackage[dvipsnames]{xcolor}

\graphicspath{{./},{./figures/}}

\DeclareMathOperator*{\argmax}{arg\:max}
\DeclareMathOperator*{\argmin}{arg\:min}
\DeclareMathOperator*{\arginf}{arg\:inf}

\begin{document}
\newtheorem{thm}{Theorem}
\newtheorem{corollary}[thm]{Corollary}
\newtheorem{conj}[thm]{Conjecture}
\newtheorem{lemma}[thm]{Lemma}
\newtheorem{proposition}[thm]{Proposition}
\newtheorem{problem}{Problem}
\newtheorem{remark}{Remark}
\newtheorem{definition}{Definition}
\newtheorem{example}{Example}

\newcommand{\prox}{\rm{prox}}
\newcommand{\bp}{{\bm{p}}}
\newcommand{\bq}{{\bm{q}}}
\newcommand{\bc}{{\bm{c}}}
\newcommand{\be}{{\bm{e}}}
\newcommand{\br}{{\bm{r}}}
\newcommand{\bw}{{\bm{w}}}
\newcommand{\bx}{{\bm{x}}}
\newcommand{\by}{{\bm{y}}}
\newcommand{\bz}{{\bm{z}}}
\newcommand{\bbf}{{\bm{f}}}
\newcommand{\brh}{{\bm{\varrho}}}
\newcommand{\balpha}{{\bm{\alpha}}}
\newcommand{\bbeta}{{\bm{\beta}}}
\newcommand{\bdelta}{{\bm{\delta}}}
\newcommand{\bPi}{{\bm{\Pi}}}
\newcommand{\bC}{{\bm{C}}}
\newcommand{\bL}{{\bm{L}}}
\newcommand{\bI}{{\bm{I}}}
\newcommand{\bP}{{\bm{P}}}
\newcommand{\bD}{{\bm{D}}}
\newcommand{\bQ}{{\bm{Q}}}
\renewcommand{\dagger}{{T}}

\renewcommand{\thefootnote}{\fnsymbol{footnote}}

\newcommand{\differential}{{\rm{d}}}

\newcommand{\interior}[1]{%
  {\kern0pt#1}^{\mathrm{o}}%
}

\newcommand{\diag}{\operatorname{diag}}
\newcommand{\tr}{\operatorname{trace}}
\newcommand{\ignore}[1]{}

\newcommand{\magenta}{\color{magenta}}
\newcommand{\red}{\color{red}}
\newcommand{\blue}{\color{blue}}
\newcommand{\gray}{\color{gray}}
\definecolor{grey}{rgb}{0.6,0.3,0.3}
\definecolor{lgrey}{rgb}{0.9,.7,0.7}

\newcommand{\symsum}{\displaystyle\sum_{\rm{symm}}}

\def\spacingset#1{\def\baselinestretch{#1}\small\normalsize}
\setlength{\parindent}{20pt}
\setlength{\parskip}{12pt}
\spacingset{1}

\title{\huge{Gradient Flow Algorithms for Density Propagation in Stochastic Systems}}

\author{Kenneth F. Caluya, and Abhishek Halder
\thanks{K.F. Caluya, and A. Halder are with the Department of Applied Mathematics, University of California, Santa Cruz,
        CA 95064, USA;
        {\texttt{kcaluya,ahalder@ucsc.edu}}}
}

\markboth{\today}{}

\maketitle

\begin{abstract}
We develop a new computational framework to solve the partial differential equations (PDEs) governing the flow of the joint probability density functions (PDFs) in continuous-time stochastic nonlinear systems. The need for computing the transient joint PDFs subject to prior dynamics arises in uncertainty propagation, nonlinear filtering and stochastic control. Our methodology breaks away from the traditional approach of spatial discretization or function approximation -- both of which, in general, suffer from the ``curse-of-dimensionality". In the proposed framework, we discretize time but not the state space. We solve infinite dimensional proximal recursions in the manifold of joint PDFs, which in the small time-step limit, is theoretically equivalent to solving the underlying transport PDEs. The resulting computation has the geometric interpretation of gradient flow of certain free energy functional with respect to the Wasserstein metric arising from the theory of optimal mass transport. We show that dualization along with an entropic regularization, leads to a cone-preserving fixed point recursion that is proved to be contractive in Thompson metric. A block co-ordinate iteration scheme is proposed to solve the resulting nonlinear recursions with guaranteed convergence. This approach enables remarkably fast computation for non-parametric transient joint PDF propagation. Numerical examples and various extensions are provided to illustrate the scope and efficacy of the proposed approach.  
\end{abstract}


\section{Introduction}\label{IntroSection}
Consider the continuous-time dynamics of the state vector $\bm{x}(t)\in\mathbb{R}^{n}$ governed by an It\^{o} stochastic differential equation (SDE) 
\begin{eqnarray}
\differential\bx = \bm{f}\left(\bm{x},t\right)\differential t \: + \: \bm{g}(\bm{x},t)\:\differential\bw, \quad \bm{x}(t=0)=\bm{x}_{0},
\label{genericItoSDE}	
\end{eqnarray}
where the joint probability density function (PDF) for the initial condition $\bm{x}_{0}$ is a known function $\rho_{0}$; we use the notation $\bm{x}_{0} \sim \rho_{0}$. The process noise $\bw(t)\in\mathbb{R}^{m}$ is Wiener and satisfy $\mathbb{E}\left[\differential w_{i} \differential w_{j}\right] = \delta_{ij}\differential t$ for all $i,j=1,\hdots,n$, where $\delta_{ij}=1$ for $i=j$, and zero otherwise. Then, the flow of the joint PDF $\rho(\bm{x},t)$ for the state vector $\bm{x}(t)$ (i.e., $\bm{x} \sim \rho$) is governed by the partial differential equation (PDE) initial value problem:
\begin{subequations}
	\begin{align}{}
			\displaystyle\frac{\partial\rho}{\partial t} = -\nabla\cdot\left(\rho\bm{f}\right) + \frac{1}{2}\sum_{i,j=1}^{n}\frac{\partial^{2}}{\partial x_{i}\partial x_{j}}(\rho\bm{g}\bm{g}^{\top})_{ij}, \label{FPKoperator}\\ 
	\qquad\rho(\bm{x},t=0) = \rho_{0}(\bm{x}) \quad \text{(given)}.
	\end{align}
\label{FPKivp}	
\end{subequations} 
The \emph{second order} transport PDE (\ref{FPKoperator}) is known as the \emph{Fokker-Planck} or \emph{Kolmogorov's forward equation} \cite{risken1996fokker}. Hereafter, we will refer it as the Fokker-Planck-Kolmogorov (FPK) PDE. 

In this paper, we consider the problem of density or belief propagation, i.e., the problem of computing the transient joint PDF $\rho(\bm{x},t)$  that solves a PDE initial value problem of the form (\ref{FPKivp}). From an application standpoint, the need for computing $\rho(\bm{x},t)$ can be motivated by two types of problems. The \emph{first} is dispersion analysis, where one is interested in predicting or analyzing the uncertainty evolution over time, e.g., in meteorological forecasting \cite{ehrendorfer1994liouville}, spacecraft entry-descent-landing \cite{halder2010beyond,halder2011dispersion}, orientation density evolution for liquid crystals in chemical physics \cite{hess1976fokker,muschik1997mesoscopic,kalmykov1998analytical}, and in motion planning \cite{park2005diffusion,park2008kinematic,hamann2008framework}. In these applications, the quantity of interest is the joint PDF $\rho(\bm{x},t)$ and its statistics (e.g., transient moments and marginal PDFs). The \emph{second} type of applications require $\rho(\bm{x},t)$ as an intermediate step toward computing other quantities of interest. For example, in nonlinear filtering \cite{challa2000nonlinear,daum2005nonlinear}, the joint PDF $\rho(\bm{x},t)$ serves as the prior in computing the posterior (i.e., conditional state) PDF. In probabilistic model validation \cite{halder2011model,halder2012further, halder2014probabilistic} and controller verification \cite{halder2015optimal}, computing $\rho(\bm{x},t)$ helps in quantifying the density-level prediction-observation mismatch.  All these applications require fast computation of $\rho(\bm{x},t)$ in a scalable and unified manner, as opposed to developing algorithms in a case-by-case basis.

Given its widespread applications, problem (\ref{FPKivp}) has received sustained attention from the scientific computing community where the predominant solution approaches have been spatial discretization and function approximation -- both of which, in general, suffer from the ``curse of dimensionality" \cite{bellman1957dynamic}. The purpose of this paper is to pursue a systems-theoretic variational viewpoint for computing $\rho(\bm{x},t)$ that breaks away from the ``solve PDE as a PDE" philosophy, and instead solves (\ref{FPKivp}) as a gradient descent on the manifold of joint PDFs. This emerging geometric viewpoint for uncertainty propagation and filtering has been reported in our recent work \cite{halder2017gradient,halder2018gradient}, but it remained unclear whether this viewpoint can offer computational benefit over the standard PDE solvers. It is not at all obvious whether and how an infinite-dimensional gradient descent can numerically obviate function approximation or spatial discretization. The contribution of this paper is to demonstrate that not only this is possible, but also that the same enables fast computation.

To conceptualize the main idea, we appeal to the \emph{metric viewpoint} of gradient descent, where a continuous-time gradient flow is realized by small time-step recursions of a proximal operator with respect to (w.r.t.) a suitable metric. For example, consider the finite dimensional gradient flow 
\begin{eqnarray}
\dfrac{\differential \bx}{\differential t } = -\nabla \varphi \left(\bx\right), \quad \bx(0) = \bx_{0}, 
\label{EGRflow}
\end{eqnarray}
where $\bx ,\bx_0\in \mathbb{R}^{n}$, and the continuously differentiable function $\varphi : \mathbb{R}^{n} \rightarrow \mathbb{R}_{\geq 0}$. The flow $\bx(t)$ generated by (\ref{EGRflow}) can be realized via variational recursion
\begin{eqnarray}
\bx_{k} = \underset{\bx\in\mathbb{R}^{n}} {{\rm{arg\:min}}}\: \frac{1}{2} \parallel \bx - \bx_{k-1} \parallel_{2}^2 + h\: \varphi(\bx) + o(h), \quad k\in\mathbb{N},
\label{ProxEGD}
\end{eqnarray}
in the sense that as the step-size $h\downarrow 0$, we have $\bm{x}_{k} \rightarrow \bm{x}(t=kh)$, i.e., the sequence $\{\bx_{k}\}$ converges pointwise to the flow $\bx(t)$. This can be verified by rewriting the Euler discretization of (\ref{EGRflow}), given by
\begin{eqnarray*}
\bx_{k} - \bx_{k-1} =-h \nabla \varphi(\bx_{k-1}), 
\end{eqnarray*} 
as
\begin{align}
\bx_{k} &= \underset{\bx\in\mathbb{R}^{n}}{{\rm{arg\:min}}}\: \frac{1}{2} \parallel \bx - \left(\bx_{k-1} - h \nabla \varphi(\bx_{k-1})\right) \parallel_{2}^2\nonumber\\
&=  \underset{\bx\in\mathbb{R}^{n}}{{\rm{arg\:min}}}\:\frac{1}{2} \parallel \bx - \bx_{k-1} \parallel_{2}^2 \:+\: \langle \bx - \bx_{k-1}, h\nabla\varphi(\bx_{k-1}) \rangle \nonumber\\
&\qquad\qquad\qquad\qquad\qquad\qquad\qquad\qquad\quad \:+\: h\varphi(\bx_{k-1}),
\label{FiniteDimGDProx}
\end{align}
where we used the fact that adding and omitting constant terms do not change the arg min. From (\ref{FiniteDimGDProx}), one can arrive at (\ref{ProxEGD}) by invoking first order approximation of $\varphi(\bx)$ at $\bx=\bx_{k}$. In (\ref{ProxEGD}), the mapping $\bx_{k-1} \mapsto \bx_{k}$ given by
\begin{eqnarray}
{\prox}^{\parallel \cdot \parallel_{2}}_{h\varphi}(\bx_{k-1}) := \underset{\bx\in\mathbb{R}^{n}} {{\rm{arg\:min}}}\: \frac{1}{2} \parallel \bx - \bx_{k-1} \parallel_{2}^2 + h\: \varphi(\bx),
\label{FiniteDimProxOpDef}
\end{eqnarray}
is called the ``proximal operator" \cite[p. 142]{parikh2014proximal} of $h\varphi$ w.r.t. the standard Euclidean metric $\parallel \cdot \parallel_{2}$. Notice that $\varphi(\cdot)$ serves as a Lyapunov function since the quantity
\begin{eqnarray}
\displaystyle\frac{\differential}{\differential t}\varphi = \langle \nabla\varphi, - \nabla\varphi\rangle = -\parallel \nabla\varphi\parallel_{2}^{2}
\label{EGRlyap}	
\end{eqnarray}
equals $0$ at the stationary point of (\ref{EGRflow}), and $<0$ otherwise.

\begin{figure}[t]
\centering
\vspace*{0.05in}
\includegraphics[width=.89\linewidth]{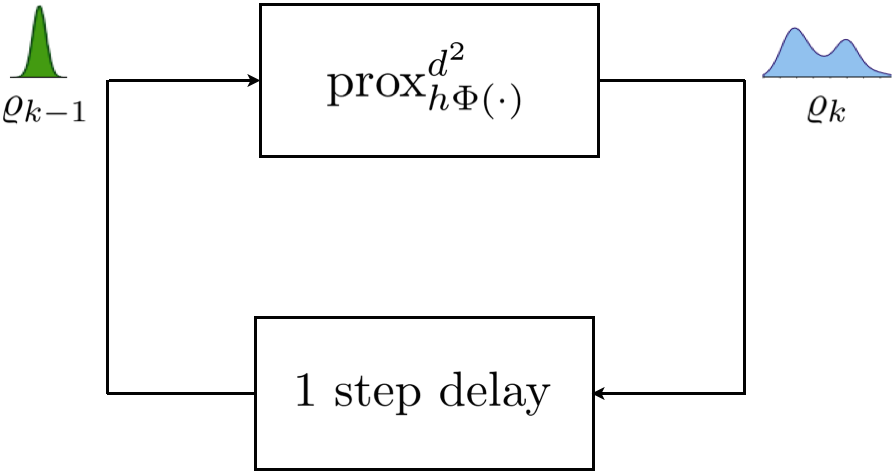}
\caption{\small{The gradient descent on the manifold of PDFs can be described by successive evaluation of proximal operators to recursively update PDFs from time $t=(k-1)h$ to $t= kh$ for $k\in\mathbb{N}$, and time-step $h>0$.}}
\vspace*{-0.25in}
\label{fig:ProxSchematic}
\end{figure}

In the infinite dimensional setting, we are interested in computing the flow generated by (\ref{FPKivp}) via gradient descent on the manifold of joint PDFs with finite second (raw) moments, denoted as\footnote{We denote the expectation operator w.r.t. the measure $\rho(\bx)\differential\bx$ as $\mathbb{E}_{\rho}\left[\cdot\right]$.} 
\[\mathscr{D}_{2} := \{\rho:\mathbb{R}^{n}\mapsto\mathbb{R} \mid \rho \geq 0, \int_{\mathbb{R}^{n}}\rho=1, \:\mathbb{E}_{\rho}[\bm{x}^{\top}\bm{x}] < \infty\}.\] 
Specifically, let $d(\cdot,\cdot)$ be a distance metric on the manifold $\mathscr{D}_{2}$, and let the functional $\Phi : \mathscr{D}_{2} \mapsto \mathbb{R}_{\geq 0}$. Then, for some chosen step-size $h>0$, the infinite dimensional proximal operator of $h\Phi$ w.r.t. the distance metric $d(\cdot,\cdot)$, given by
\begin{eqnarray}
{\prox}^{d(\cdot)}_{h\Phi}(\varrho_{k-1}) := \underset{\varrho\in\mathscr{D}_{2}} {{\rm{arg\:inf}}}\: \frac{1}{2} d\left(\varrho,\varrho_{k-1}\right)^2 + h\: \Phi(\varrho),
\label{InfiniteDimProxOpDef}
\end{eqnarray}
can be used to define a proximal recursion (Fig. \ref{fig:ProxSchematic}):
\begin{eqnarray}
\varrho_{k} = {\prox}^{d(\cdot)}_{h\Phi}(\varrho_{k-1}), \quad k\in\mathbb{N}, \quad \varrho_{0}(\bm{x}) := \rho_{0}(\bm{x}).
\label{ProxRecursionInfiniteDim}	
\end{eqnarray}
Just as the proximal recursion (\ref{ProxEGD}) approximates the finite dimensional flow (\ref{EGRflow}), similarly it is possible to design $d(\cdot,\cdot)$ and $\Phi(\cdot)$ in (\ref{InfiniteDimProxOpDef}) as function of the drift $\bm{f}$ and diffusion $\bm{g}$ in (\ref{FPKoperator}) such that the proximal recursion (\ref{ProxRecursionInfiniteDim}) approximates the infinite dimensional flow (\ref{FPKivp}). This idea was first proposed in \cite{jordan1998variational}, showing that when $\bm{f}$ is a gradient vector field and $\bm{g}$ is a scalar multiple of identity matrix, then the distance $d(\cdot,\cdot)$ can be taken as the Wasserstein-2 metric \cite{villani2003topics} with $\Phi(\cdot)$ as the free energy functional. In particular, the solution of (\ref{ProxRecursionInfiniteDim}) was shown to converge to the flow of (\ref{FPKivp}), i.e., $\varrho_{k}(\bx) \rightarrow \rho(\bx, t=kh)$ in strong $L_{1}(\mathbb{R}^{n})$ sense, as $h\downarrow 0$. The resulting variational recursion (\ref{ProxRecursionInfiniteDim}) has since been known as the \emph{Jordan-Kinderlehrer-Otto (JKO) scheme} \cite{ambrosio2008gradient}, and we will refer to the FPK operator (\ref{FPKoperator}) with such assumptions on $\bm{f}$ and $\bm{g}$ to be in ``JKO canonical form". Similar gradient descent schemes have been derived for many other PDEs; see e.g., \cite{santambrogio2017euclidean} for a recent survey. 

The remaining of this paper is organized as follows. Section \ref{JKOformsectionlabel} explains the JKO canonical form and the corresponding free energy functional $\Phi(\cdot)$. Our algorithms and convergence results are presented in Section \ref{MainResultSectionLabel}, followed by numerical simulation results in Section \ref{NumericalResultsSectionLabel}. Section \ref{ExtensionsSectionLabel} provides various extensions of the basic algorithm showing how the framework proposed here can be applied to systems not in JKO canonical form. Section \ref{ConclusionsSectionLabel} concludes the paper.

We remark here that a preliminary version \cite{caluya2018proximal} of this work appeared in the 2019 American Control Conference. This paper significantly expands \cite{caluya2018proximal} by incorporating additional results for the so-called McKean-Vlasov flow (Sections \ref{subsecMVintro} and \ref{SubsecMV}), and by providing various extensions of the basic algorithm (Section \ref{ExtensionsSectionLabel}). 

\subsubsection{Notations} 
Throughout the paper, we use bold-faced capital letters for matrices and bold-faced lower-case letters for column vectors. We use the symbol $\langle \cdot, \cdot\rangle$ to denote the Euclidean inner product. In particular, $\langle \bm{A},\bm{B} \rangle:= {\mathrm{trace}}({\bm{A}}^{\top}\bm{B})$ denotes
Frobenius inner product between matrices $\bm{A}$ and $\bm{B}$, and $\langle \bm{a},\bm{b}\rangle := \bm{a}^{\top}\bm{b}$ denotes the inner product between column vectors $\bm{a}$ and $\bm{b}$. The symbols $\nabla$ and $\Delta$ denote the (Euclidean) gradient and the Laplacian operators, respectively. We use $\mathcal{N}(\mu,\sigma^2)$ to denote a univariate Gaussian PDF with mean $\mu$ and variance $\sigma^2$. Likewise, $\mathcal{N}(\bm{\mu},\bm{\Sigma})$ denotes a multivariate Gaussian PDF with mean vector $\bm{\mu}$ and covariance matrix $\bm{\Sigma}$. The expectation operator w.r.t. the PDF $\rho$ is denoted by $\mathbb{E}_{\rho}\left[\cdot\right]$, i.e., $\mathbb{E}_{\rho}\left[\cdot\right]:=\int_{\mathbb{R}^{n}}(\cdot)\rho\:\differential\bx$. The operands $\log(\cdot)$, $\exp(\cdot)$ and $\geq 0$ are to be understood as element-wise. The notations $\odot$ and $\oslash$ denote element-wise (Hadamard) product and division, respectively. We use $\bm{I}_{n}$ to denote the $n\times n$ identity matrix. The symbols $\bm{1}$ and $\bm{0}$ stand for column vectors of appropriate dimension containing all ones, and all zeroes, respectively.  

\subsubsection{Preliminaries}
We next collect definitions and some properties of the 2-Wasserstein metric and the Kullback-Leibler divergence, which will be useful in the development below. 
\begin{definition}(\textbf{2-Wasserstein metric})\label{WdefnLabel}
The 2-Wasserstein metric between two probability measures $\differential\pi_{1}(\bx) := \rho_{1}(\bx)\differential\bx$ and 	$\differential\pi_{2}(\bm{y}) := \rho_{2}(\bm{y})\differential\bm{y}$, supported respectively on $\mathcal{X},\mathcal{Y}\subseteq\mathbb{R}^{n}$, is denoted as $W(\pi_{1},\pi_{2})$ (equivalently, $W(\rho_{1},\rho_{2})$ whenever the measures $\pi_{1},\pi_{2}$ are absolutely continuous so that the respective PDFs $\rho_{1},\rho_{2}$ exist); it is defined as
\begin{eqnarray}
W(\pi_{1},\pi_{2}):=  \left(\underset{\differential\pi\in\Pi\left(\pi_{1},\pi_{2}\right)}{\inf}\displaystyle\int_{\mathcal{X}\times\mathcal{Y}}s(\bm{x};\bm{y})\:\differential\pi\left(\bm{x},\bm{y}\right) \right)^{\frac{1}{2}},
\label{Wdefn}	
\end{eqnarray}
where $s(\bm{x};\bm{y}):=\parallel \bm{x} - \bm{y} \parallel_{2}^{2}$ is the squared Euclidean distance in $\mathbb{R}^{n}$, and $\Pi\left(\pi_{1},\pi_{2}\right)$ denotes the collection of all joint probability measures on $\mathcal{X}\times\mathcal{Y}$ having finite second moments, with marginals $\pi_{1}$ and $\pi_{2}$, respectively.
\end{definition}
The existence and uniqueness of the minimizer in (\ref{Wdefn}) is guaranteed. It is well-known \cite[Ch. 7]{villani2003topics} that $W(\pi_{1},\pi_{2})$ defines a metric on the manifold $\mathscr{D}_{2}$. This means that $W(\pi_{1},\pi_{2}) \geq 0$ with equality if and only if $\pi_{1}=\pi_{2}$, the symmetry: $W(\pi_{1},\pi_{2})=W(\pi_{2},\pi_{1})$, and that $W(\pi_{1},\pi_{2})$ satisfies the triangle inequality. Its square, $W^{2}(\pi_{1},\pi_{2})$ equals \cite{benamou2000computational} the minimum amount of work required to transport $\pi_{1}$ to $\pi_{2}$ (or equivalently, $\rho_{1}$ to $\rho_{2}$), and vice versa. For any PDF $\nu$, the function $\rho \mapsto W^{2}(\rho,\nu)$ is convex on $\mathscr{D}_{2}$, i.e., for any $\rho_{1},\rho_{2} \in \mathscr{D}_{2}$, and $0\leq \tau \leq 1$, we have
\begin{eqnarray}
W^{2}(\nu,(1-\tau)\rho_{1} + \tau\rho_{2}) \:\leq\:(1-\tau)W^{2}(\nu,\rho_{1}) + \tau W^{2}(\nu,\rho_{2}).
\label{IndividualConvexity}	
\end{eqnarray}

\begin{definition}(\textbf{Kullback-Leibler divergence})\label{KLdivLabel}
The Kullback-Leibler divergence, also known as relative entropy, between two probability measures $\differential\pi_{1}(\bx) := \rho_{1}(\bx)\differential\bx$ and 	$\differential\pi_{2}(\bm{y}) := \rho_{2}(\bm{y})\differential\bm{y}$, denoted as $D_{\rm{KL}}(\pi_{1} || \pi_{2})$, is defined as
\begin{align}
D_{\rm{KL}}(\pi_{1} || \pi_{2}) &= \displaystyle\int \left(\displaystyle\frac{\differential\pi_{1}}{\differential\pi_{2}}\right)\log\left(\displaystyle\frac{\differential\pi_{1}}{\differential\pi_{2}}\right)\differential\pi_{2} \notag\\
&= \displaystyle\int_{\mathbb{R}^{n}}\rho_{1}(\bx)\log\displaystyle\frac{\rho_{1}(\bx)}{\rho_{2}(\bx)}\differential\bx,\label{KLdefn}
\end{align}
where $\frac{\differential\pi_{1}}{\differential\pi_{2}}$ denotes the Radon-Nikodym derivative. Henceforth, we use the notational equivalence $D_{\rm{KL}}(\pi_{1} || \pi_{2}) \equiv D_{\rm{KL}}(\rho_{1} || \rho_{2})$.
\end{definition}
From Jensen's inequality, $D_{\rm{KL}}(\rho_{1} || \rho_{2}) \geq 0$; however, $D_{\rm{KL}}$ is not a metric since it is neither symmetric, nor does it satisfy the triangle inequality. The Kullback-Leibler divergence (\ref{KLdefn}) is jointly convex in $\rho_{1}$ and $\rho_{2}$.


\section{JKO Canonical Form}\label{JKOformsectionlabel}
The JKO canonical form refers to a continuous-time stochastic dynamics where the drift vector field is the gradient of a potential function. The potential can be state dependent, or can depend on both the state $\bx(t)$ and its joint PDF $\rho(\bx,t)$. In the former case, the associated flow of the joint state PDF is governed by a FPK PDE of the form (\ref{FPKoperator}) whereas in the latter case, the same is governed by the McKean-Vlasov integro-PDE. For both cases, the proposed gradient flow algorithms (Section \ref{MainResultSectionLabel}) will be able to compute the transient state PDFs $\rho(\bx,t)$. We next describe these canonical forms; generalization of our framework to systems \emph{not} in JKO canonical form will be given in Sections \ref{NumericalResultsSectionLabel} and \ref{ExtensionsSectionLabel}.

\subsection{FPK Gradient Flow}

Consider the It\^{o} SDE
\begin{eqnarray}
\differential\bx = -\nabla \psi\left(\bx\right)\differential t \: + \: \sqrt{2\beta^{-1}}\:\differential\bw	, \quad \bx(0) = \bx_{0},
\label{ItoGradient}
\end{eqnarray}
where the the drift potential $\psi : \mathbb{R}^{n} \mapsto (0,\infty)$, the diffusion coefficient $\beta > 0$, and the initial condition $\bx_{0} \sim \rho_{0}(\bx)$. For the sample path $\bx(t)$ dynamics given by the SDE (\ref{ItoGradient}), the flow of the joint PDF $\rho\left(\bm{x},t\right)$ is governed by the FPK PDE initial value problem
\begin{eqnarray}
\displaystyle\frac{\partial\rho}{\partial t} = \nabla\cdot\left(\rho\nabla \psi\right) \: + \: \beta^{-1}\Delta\rho, \quad \rho(\bx,0) = \rho_{0}(\bx).
\label{FPKgradient}	
\end{eqnarray}
It is easy to verify that the unique stationary solution of (\ref{FPKgradient}) is the Gibbs PDF $\rho_{\infty}(\bm{x}) = \kappa \exp\left(-\beta \psi(\bm{x})\right)$, where the normalizing constant $\kappa$ is known as the \emph{partition function}. 

A Lyapunov functional associated with (\ref{FPKgradient}) is the \emph{free energy}
\begin{eqnarray}
F(\rho) &:=& \mathbb{E}_{\rho}\left[\psi + \beta^{-1}\log\rho \right] \nonumber\\
&=& \beta^{-1} D_{{\rm{KL}}}\left(\rho \parallel \exp\left(-\beta \psi(\bm{x})\right)\right)	\geq 0,
\label{FPKFreeEnergy}
\end{eqnarray}
that decays \cite{jordan1998variational} along the solution trajectory of (\ref{FPKgradient}). This follows from re-writing (\ref{FPKgradient}) as
\begin{eqnarray}
\displaystyle\frac{\partial{\rho}}{\partial t} = \nabla\cdot\left(\rho\nabla\zeta\right), \quad \text{where}\quad \zeta := \beta^{-1}\left(1 + \log\rho\right) + \psi,
\label{FPKadvectionform}	
\end{eqnarray}
and consequently
\begin{eqnarray}
\frac{\differential}{\differential t}F = -\mathbb{E}_{\rho}\left[\parallel \nabla\zeta \parallel_{2}^{2}\right],
\label{dFdt}	
\end{eqnarray}
which is $<0$ for the transient solution $\rho(\bx,t)$, and $=0$ at the stationary solution $\rho_{\infty}(\bx) = \kappa \exp(-\beta \psi(\bm{x}))$. In our context, (\ref{dFdt}) is an analogue of (\ref{EGRlyap}). Notice that the free energy (\ref{FPKFreeEnergy}) can be seen as the sum of the \emph{potential energy} $\int_{\mathbb{R}^{n}}\psi(\bx)\rho\:\differential\bx$ and the \emph{internal energy} $\beta^{-1}\int_{\mathbb{R}^{n}}\rho\log\rho\:\differential\bx$. If $\psi\equiv 0$, the PDE (\ref{FPKgradient}) reduces to the heat equation, which by (\ref{FPKFreeEnergy}), can then be interpreted as an entropy maximizing flow.

The seminal result of \cite{jordan1998variational} was that the transient solution of (\ref{FPKgradient}) can be computed via the proximal recursion (\ref{ProxRecursionInfiniteDim}) with the distance metric $d\equiv W$ (i.e., the 2-Wasserstein metric in (\ref{Wdefn})) and the functional $\Phi \equiv F$ (i.e., the free energy (\ref{FPKFreeEnergy})). Just as (\ref{EGRflow}) and (\ref{ProxEGD}) form a gradient flow-proximal recursion pair, likewise (\ref{FPKgradient}) and (\ref{ProxRecursionInfiniteDim}) form the same with the stated choices of $d$ and $\Phi$. From (\ref{Wdefn}), notice that the distance metric $W$ itself is defined as the solution of an optimization problem, hence it is not apparent how to numerically implement the recursion (\ref{ProxRecursionInfiniteDim}) in a scalable manner.


\subsection{McKean-Vlasov Gradient Flow}\label{subsecMVintro}
In addition to the drift and diffusion, if one has non-local interaction, then the corresponding PDF evolution equation becomes the McKean-Vlasov integro-PDE   
\begin{eqnarray}
\displaystyle\frac{\partial\rho}{\partial t} = \nabla\cdot\left(\rho\nabla\left(\psi + \rho * \phi\right)\right) \: + \: \beta^{-1}\Delta\rho, \quad \rho(\bx,0) = \rho_{0}(\bx),
\label{MVFPKgradient}	
\end{eqnarray}
where $*$ denotes the convolution in $\mathbb{R}^{n}$, the interaction potential $\phi : \mathbb{R}^{n} \mapsto (0,\infty)$ and is symmetric, i.e., $\phi(-\bm{v}) = \phi(\bm{v})$ for $\bm{v}\in\mathbb{R}^{n}$. The associated sample path $\bm{x}(t)$ dynamics has PDF-dependent drift:
\begin{eqnarray}
\differential\bx = -\left(\nabla \psi\left(\bx\right) + \nabla\left(\rho * \phi\right) \right)\differential t + \sqrt{2\beta^{-1}}\:\differential\bw, \quad \bx(0) = \bx_{0}.
\label{MVsde}	
\end{eqnarray}
As an example, when $\phi(\bm{v}):=\frac{1}{2}\parallel\bm{v}\parallel^{2}$, then $\nabla(\rho * \phi)(\bx) = \bm{x} - \mathbb{E}_{\rho}[\bm{x}]$.
Clearly, (\ref{MVFPKgradient}) reduces to (\ref{FPKgradient}) in the absence of interaction ($\phi\equiv 0$). The McKean-Vlasov equation serves as a model for coupled multi-agent interaction in applications such as crowd movement \cite{colombo2012nonlocal}, opinion dynamics \cite{gomez2012bounded,askarzadeh2019}, population biology and communication systems.

A Lyapunov functional for (\ref{MVFPKgradient}) can be obtained \cite{villani2004trend} by generalizing the free energy (\ref{FPKFreeEnergy}) as
\begin{eqnarray}
F(\rho) := \mathbb{E}_{\rho}\left[\psi \: + \: \beta^{-1}\log\rho \: + \: \rho * \phi\right],
\label{MVFreeEnergy}	
\end{eqnarray}
which is a sum of the potential energy (as before), the internal energy (as before), and the \emph{interaction energy} $\frac{1}{2}\int_{\mathbb{R}^{2n}}\phi(\bx-\by)\rho(\bx)\rho(\by)\:\differential\bx\differential\by$. In this case, (\ref{dFdt}) holds with 
\begin{eqnarray}
	\zeta := \beta^{-1}\left(1 + \log\rho\right) + \psi + \rho * \phi,
	\label{zetaforMV}
\end{eqnarray}
which is often referred to as the \emph{entropy dissipation functional} \cite{carrillo2003kinetic}. As before, the proximal recursion (\ref{ProxRecursionInfiniteDim}) with $d\equiv W$ and $\Phi \equiv F$ (now $F$ given by (\ref{MVFreeEnergy})), approximates the flow (\ref{MVFPKgradient}).

\section{Framework}\label{MainResultSectionLabel}
We now describe our computational framework to solve the proximal recursion
\begin{subequations}
\begin{align}{}
\varrho_{k} &= {\prox}^{W}_{hF(\cdot)} (\varrho_{k-1}) \label{JKOscheme:a} \\
 &=  \underset{\varrho \in \mathscr{D}_2 }\arginf \ \frac{1}{2} W^2(\varrho_{k-1},\varrho) + h\: F(\varrho), \quad k\in\mathbb{N}, \label{JKOscheme:b}
 \end{align}
 \label{JKOscheme}
\end{subequations}
with $\varrho_{0} \equiv \rho_{0}(\bx)$ (the initial joint PDF) for some small step-size $h$. For maximal clarity, we develop the framework with the free energy $F(\cdot)$ as in (\ref{FPKFreeEnergy}), i.e., for FPK gradient flow. In Section \ref{NumericalResultsSectionLabel}.C, we will show how the same can be generalized when $F(\cdot)$ is of the form (\ref{MVFreeEnergy}). As per the convexity properties mentioned in Section \ref{IntroSection}.2, problem (\ref{JKOscheme}) involves (recursively) minimizing sum of two convex functionals, and hence is a convex problem for each $k\in\mathbb{N}$.

We discretize time as $t=0,h,2h, \hdots$, and develop an algorithm to solve (\ref{JKOscheme}) without making any spatial discretization. Specifically, we would like to perform the recursion (\ref{JKOscheme}) on weighted scattered point cloud $\{\bx_{k}^{i},\varrho_{k}^{i}\}_{i=1}^{N}$ of cardinality $N$ at $t_{k} = kh$, $k\in\mathbb{N}$, where the location of the $i$\textsuperscript{th} point $\bx_{k}^{i}\in\mathbb{R}^{n}$ denotes its state-space coordinate, and the corresponding weight $\varrho_{k}^{i} \in \mathbb{R}_{\geq 0}$ denotes the value of the joint PDF evaluated at that point at time $t_{k}$. Such weighted point cloud representation of (\ref{JKOscheme}) results in the following problem: 
\begin{align}
 \brh_k  =\underset{\brh}\argmin \bigg\{  \underset{\bm{M }\in \Pi(\brh_{k-1},\brh)} \min \frac{1}{2}\langle \bm{C}_{k},\bm{M}\rangle + h \: \langle \bm{\psi}_{k-1}\notag\\
 +\beta^{-1}\log\brh,\brh \rangle   \bigg\}, \quad k\in\mathbb{N},
 \label{FiniteSampleJKO}
\end{align}
where the drift potential vector $\bm{\psi}_{k-1} \in \mathbb{R}^{N}$ is given by
\begin{eqnarray}
\bm{\psi}_{k-1}(i) := \psi\left(\bm{x}_{k-1}^{i}\right), \quad i=1,\hdots,N.
\label{Discretepsidef}	
\end{eqnarray}
Here, the probability vectors $\brh, \brh_{k-1} \in \mathcal{S}_{N-1}$, the probability simplex in $\mathbb{R}^{N}$. Furthermore, for each $k\in\mathbb{N}$, the matrix $\bm{C}_{k}\in\mathbb{R}^{N\times N}$ is given by
\begin{eqnarray}
\bm{C}_{k}(i,j) := s(\bm{x}_{k}^{i};\bm{x}_{k-1}^{j}) = \parallel \bm{x}_{k}^{i} -  \bm{x}_{k-1}^{j} \parallel_{2}^{2}, \; i,j=1,\hdots,N,
\label{EDMdef}	
\end{eqnarray}
and $\Pi(\brh_{k-1},\brh)$ stands for the set of all matrices $\bm{M}\in\mathbb{R}^{N\times N}$ such that 
\begin{align}
\bm{M}\geq 0, \quad \bm{M}\bm{1} = \brh_{k-1}, \quad \bm{M}^{\top}\bm{1} = \brh.
\label{CouplingMatrixConstr}
\end{align}
Notice that the inner minimization in  (\ref{FiniteSampleJKO}) is a standard linear programming problem if it were to be solved for a given $\brh\in\mathcal{S}_{N-1}$, as in the Monge-Kantorovich optimal mass transport \cite{villani2003topics}. However, the outer minimization in  (\ref{FiniteSampleJKO}) precludes a direct numerical approach.

To circumvent the aforesaid issues, following \cite{karlsson2017generalized}, we first regularize and then dualize (\ref{FiniteSampleJKO}). Specifically, adding an entropic regularization $H(\bm{M}) := \langle\bm{M}, \log\bm{M}\rangle$ in (\ref{FiniteSampleJKO}) yields
\begin{align}
\brh_k = \underset{\brh}\argmin \bigg\{  \underset{\bm{M} \in \Pi(\brh_{k-1},\brh)} \min \frac{1}{2}\langle \bm{C}_{k},\bm{M}\rangle +
\epsilon H(\bm{M}) \notag\\
+ h \: \langle \bm{\psi}_{k-1}+\beta^{-1}\log\brh,\brh \rangle   \bigg\},
\label{EntropyRegJKO}
\end{align}
where $\epsilon>0$ is a regularization parameter. The entropic regularization is standard in optimal mass transport literature \cite{cuturi2013sinkhorn,benamou2015iterative} and leads to efficient Sinkhorn iteration for the inner minimization. Here we point out that there has been parallel work \cite{chen2016entropic,chen2016relation} on the connection between optimal mass transport and the so called Schr\"{o}dinger bridge problem which is a dynamic version of this type of regularization.

In our context, the entropic regularization ``algebrizes" the inner minimization in the sense if $\bm{\lambda}_{0},\bm{\lambda}_{1}\in\mathbb{R}^{N}$ are Lagrange multipliers associated with the equality constraints in (\ref{CouplingMatrixConstr}), then the optimal coupling matrix $\bm{M}^{\rm{opt}} := [m^{\rm{opt}}(i,j)]_{i,j=1}^{N}$ in (\ref{EntropyRegJKO}) has the Sinkhorn form
\begin{align}
m^{\rm{opt}}(i,j) = \exp\left(\bm{\lambda}_{0}(i)h/\epsilon\right) \exp\left(-\bm{C}_{k}(i,j)/(2\epsilon)\right) \notag\\
\exp\left(\bm{\lambda}_{1}(j)h/\epsilon\right).
\label{SinkhornFormInnerArgmin}	
\end{align}
Since the objective in (\ref{EntropyRegJKO}) is proper convex and lower semi-continuous in $\brh$, the \emph{strong duality} holds, and we consider the Lagrange dual of (\ref{EntropyRegJKO}) given by:
\begin{align}
\left(\bm{\lambda}_0^{\rm{opt}},\bm{\lambda}_1^{\rm{opt}}\right) 
=\underset{\bm{\lambda}_0,\bm{\lambda}_1\in\mathbb{R}^{N}}\argmax \bigg \{  \langle \bm{\lambda}_0,\brh_{k-1}  \rangle 
- F^{\star}(-\bm{\lambda}_1) \notag\\
-\frac{\epsilon }{h} \bigg( \exp({\bm{\lambda}}_{0}^{\top} h / \epsilon ) \exp(-\bm{C}_{k}/2 \epsilon) \exp({\bm{\lambda}}_1h / \epsilon )\bigg)\bigg \}, 
\label{DualJKO}
\end{align} 
where 
\begin{eqnarray}
F^{\star}(\cdot) := \underset{\vartheta}{\sup}\:\{\langle \cdot, \vartheta\rangle \: - \: F(\vartheta)\}
\label{LFtransform}	
\end{eqnarray}
is the \emph{Legendre-Fenchel conjugate} of the free energy $F$ in (\ref{FPKFreeEnergy}). 
Next, we derive the first order optimality conditions for (\ref{DualJKO}) resulting in the proximal updates, and then provide an algorithm to solve the same.

\subsection{Proximal Recursions}
Given the vectors $\brh_{k-1}, \bm{\psi}_{k-1}$, the matrix $\bm{C}_{k}$, and the positive scalars $\beta, h, \epsilon$ in (\ref{DualJKO}), let
\begin{alignat}{3}
&\bm{y}:=\exp({\bm{\lambda}}_0h / \epsilon ), \quad &&\bm{z}:=\exp({\bm{\lambda}}_{1}h / \epsilon ), \label{yzmaps}\\
&\bm{\Gamma}_{k} := {\exp(-\bm{C}_{k}/2\epsilon}), \quad &&\bm{\xi}_{k-1} := \exp(-\beta \bm{\psi}_{k-1} - \bm{1}). \label{Gammaxidef}
\end{alignat}
The following result establishes a system of nonlinear equations for computing $\bm{\lambda}_0^{\rm{opt}},\bm{\lambda}_1^{\rm{opt}}$ in (\ref{DualJKO}), and consequently $\brh_{k}$ in (\ref{EntropyRegJKO}).
\begin{thm}\label{FixedPtRecursionThm}
The vectors $\bm{\lambda}_0^{\rm{opt}},\bm{\lambda}_1^{\rm{opt}}$ in (\ref{DualJKO}) can be found by solving for $\bm{y},\bm{z}\in\mathbb{R}^{N}$ from the following system of equations:
\begin{subequations}
\begin{align}
\bm{y}\odot \left(\bm{\Gamma}_{k}\bm{z}\right) &=\bm{\rho}_{k-1}, \label{FixedPtRecursion:a}\\
\bm{z}\odot \left({\bm{\Gamma}_{k}}^{\top} \bm{y}\right) &=\bm{\xi}_{k-1}\odot \bm{z}^{-\frac{\beta\epsilon}{h}}, \label{FixedPtRecursion:b}
\end{align}
\label{FixedPtRecursion}
\end{subequations}
and then inverting the maps (\ref{yzmaps}). Let $\left(\bm{y}^{\rm{opt}},\bm{z}^{\rm{opt}}\right)$ be the solution of (\ref{FixedPtRecursion}). The vector $\brh_{k}$ in (\ref{EntropyRegJKO}), i.e., the proximal update (Fig. \ref{fig:ProxSchematic}) can then be obtained as
\begin{align}
\brh_{k} = \bm{z}^{\rm{opt}}\odot \left({\bm{\Gamma}_{k}}^{\top} \bm{y}^{\rm{opt}}\right).
\label{ProxUpdate}
\end{align}
\end{thm} 
\begin{proof}
From (\ref{FPKFreeEnergy}), the ``discrete free energy" is
\[F(\bm{\brh}) = \langle \bm{\psi}+\beta^{-1}\log\brh,\brh \rangle.\] 
Its Legendre-Fenchel conjugate, by (\ref{LFtransform}), is given by
\begin{align}
F^{\star}(\bm{\lambda}) = \underset{\brh} \sup  \big \{{ \bm{\lambda}^{\top} \bm{\brh} -  \bm{\psi}^{\top} \bm{\brh}-\beta^{-1} \bm{\brh}^{\top} \log \bm{\brh} } \big \}. 
\label{DiscreteLF}
\end{align}
Setting the gradient of the objective function in $ (\ref{DiscreteLF})$ w.r.t. $\brh$ to zero, and solving for $\brh$ yields
\begin{align}
\brh_{\max} =  \exp(\beta(\bm{\lambda}-\bm{\psi})-\bm{1}). 
\label{rhomax}
\end{align}
Substituting (\ref{rhomax}) back into (\ref{DiscreteLF}), results
\begin{align}
F^{\star}(\bm{\lambda}) = \beta^{-1}\bm{1}^{\top} \exp(\beta(\bm{\lambda}-\bm{\psi})-\bm{1}). 
\label{explicitLegendre}
\end{align}
Fixing $\bm{\lambda}_1$, and taking the gradient of the objective in (\ref{DualJKO}) w.r.t. $\bm{\lambda}_0$, gives (\ref{FixedPtRecursion:a}). Likewise, fixing $\bm{\lambda}_0$, and taking the gradient of the objective in (\ref{DualJKO}) w.r.t. $\bm{\lambda}_1$ gives 
\begin{align}
\nabla_{\bm{\lambda}_1} F^{\star}(-{\bm{\lambda}_1} )=\bm{z}\odot \left({\bm{\Gamma}_{k}}^{\top} \bm{y}\right). 
\label{gradrecur}
\end{align}
Using (\ref{explicitLegendre}) to simplify the left-hand-side of (\ref{gradrecur}) results in (\ref{FixedPtRecursion:b}). To derive (\ref{ProxUpdate}), notice that combining the last equality constraint in (\ref{CouplingMatrixConstr}) with (\ref{SinkhornFormInnerArgmin}), (\ref{yzmaps}) and (\ref{Gammaxidef}) gives \[\brh_{k} = (\bm{M}^{\rm{opt}})^{\top}\bm{1} = \sum_{j=1}^{N} m^{\rm{opt}}(j,i) = \bm{z}(i)\sum_{j=1}^{N}\bm{\Gamma}_{k}(j,i)\bm{y}(j),\] 
which is equal to $\bm{z}\odot\bm{\Gamma}_{k}^{\top}\bm{y}$, as claimed.
\end{proof}
In the following (Section \ref{SubsecAlgorithm}), we propose an algorithm to solve (\ref{FixedPtRecursion}) and (\ref{ProxUpdate}), and then outline the overall implementation of our computational framework. The convergence results for the proposed algorithm are given in Section \ref{SubsecConvergence}. 

\subsection{Algorithm} \label{SubsecAlgorithm}

\subsubsection{Proximal algorithm}

We now propose a block co-ordinate iteration scheme to solve (\ref{FixedPtRecursion}). The proposed procedure, which we call \textproc{ProxRecur}, and detail in Algortihm \ref{ProxRecur}, takes $\brh_{k-1}$ as input and returns the proximal update $\brh_{k}$ as output for $k\in\mathbb{N}$. In addition to the data $\brh_{k-1}, \bm{\psi}_{k-1}, \bm{C}_{k}, \beta, h, \epsilon, N$, the Algorithm \ref{ProxRecur} requires two parameters as user input: numerical tolerance $\delta$, and maximum number of iterations $L$. The computation in Algorithm \ref{ProxRecur}, as presented, involves making an initial guess for the vector $\bm{z}$ and then updating $\bm{y}$ and $\bm{z}$ until convergence. The iteration over index $\ell \leq L$ is performed while keeping the physical time ``frozen".
\begin{algorithm}
    \caption{Proximal recursion to compute $\brh_{k}$ from $\brh_{k-1}$}
    \label{ProxRecur}
    \begin{algorithmic}[1] 
        \Procedure{ProxRecur}{$\brh_{k-1}$, $\bm{\psi}_{k-1}$, $\bm{C}_{k}$, $\beta$, $h$, $\epsilon$, $N$, $\delta$, $L$}
            \State $\bm{\Gamma}_{k}\gets {\exp(-\bm{C}_{k}/2\epsilon})$
            \State $\bm{\xi} \gets \exp(-\beta \bm{\psi}_{k-1} - \bm{1})$
            
            \State $\bm{z}_{0} \gets {\rm{rand}}_{N\times 1}$\Comment{initialize}
            \State $\bm{z} \gets \left[\bm{z}_{0}, \bm{0}_{N\times (L-1)}\right]$ 
            \State $\bm{y} \gets \left[\brh_{k-1} \oslash \left(\bm{\Gamma}_{k}\bm{z}_{0}\right), \bm{0}_{N\times (L-1)}\right]$
            \State $\ell=1$ \Comment{iteration index}
            
            \While{$\ell \leq L$}
                \State $\bm{z}(:,\ell+1) \gets \left(\bm{\xi}_{k-1} \oslash \left(\bm{\Gamma}_{k}^{\top}\bm{y}(:,\ell)\right)\right)^{\frac{1}{1+\beta\epsilon/h}}$
                \State $\bm{y}(:,\ell+1) \gets \bm{\brh}_{k-1} \oslash \left( \bm{\Gamma}_{k} \bm{z}(:,\ell+1)\right) $
                \If{$\parallel \bm{y}(:,\ell+1) - \bm{y}(:,\ell)\parallel < \delta \;\And\; \parallel \bm{z}(:,\ell+1)- \bm{z}(:,\ell)\parallel < \delta$ } \Comment{error within tolerance}
                \State break
                \Else
                \State $\ell \gets \ell + 1$
                \EndIf
            \EndWhile\label{euclidendwhile}
            \State \textbf{return} $\brh_{k} \gets \bm{z}(:,\ell) \odot \left( \bm{\Gamma}_{k}^{\top} \bm{y}(:,\ell) \right)$\Comment{proximal update}
        \EndProcedure
    \end{algorithmic}
\end{algorithm}
\vspace*{-0.1in}

We next outline the overall algorithmic setup to implement the proximal recursion over probability weighted scattered data.

\begin{figure}[h]
\centering
\includegraphics[width=.65\linewidth]{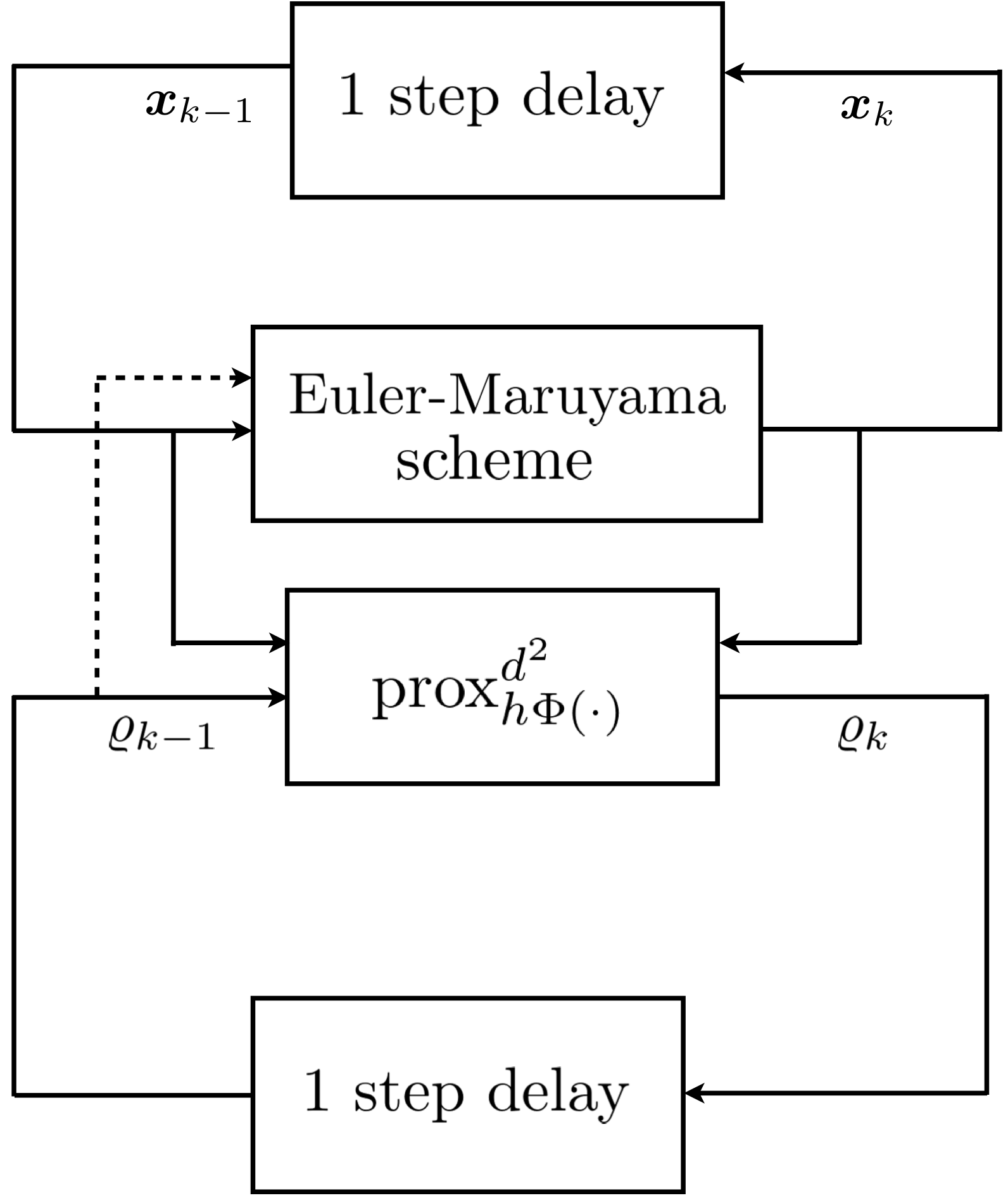}
\caption{\small{Schematic of the proposed algorithmic setup for propagating the joint state PDF as probability weighted scattered point cloud $\{\bx_{k}^{i},\varrho_{k}^{i}\}_{i=1}^{N}$. The location of the points $\{\bx_{k}^{i}\}_{i=1}^{N}$ are updated via Euler-Maruyama scheme; the corresponding probability weights are updated via Algorithm \ref{ProxRecur}. The dashed arrow shown above is present only when the state dynamics is density dependent, as in (\ref{MVsde}).
}}
\vspace*{-0.1in}
\label{BlockDiagm}
\end{figure}

\subsubsection{Overall algorithm}

Samples from the \emph{known} initial joint PDF $\rho_{0}$ are generated as point cloud $\{\bx_{0}^{i},\varrho_{0}^{i}\}_{i=1}^{N}$. Then for $k\in\mathbb{N}$, the point clouds $\{\bx_{k}^{i},\varrho_{k}^{i}\}_{i=1}^{N}$ are updated as shown in Fig. \ref{BlockDiagm}. Specifically, the state vectors are updated via Euler-Maruyama scheme applied to the underlying SDE (\ref{ItoGradient}) or (\ref{MVsde}). Explicitly, the Euler-Maruyama update for (\ref{MVsde}) is
\begin{eqnarray}
\bx_{k}^{i} = \bx_{k-1}^{i}-h\nabla\left(\psi(\bx_{k-1}^{i}) + \omega(\bx_{k-1}^{i})\right)\nonumber\\
+ \sqrt{2\beta^{-1}}\left(\bm{w}_{k}^{i} - \bm{w}_{k-1}^{i}\right),
\label{EulerMaruyamaMV}	
\end{eqnarray}
where $\omega(\cdot) := \int\phi(\cdot - \bm{y})\rho(\bm{y})\differential\bm{y}$, and $\bm{w}_{k}^{i}:=\bm{w}^{i}(t=kh)$, $k\in\mathbb{N}$. The same for (\ref{ItoGradient}) is obtained by setting $\phi \equiv \omega \equiv 0$ in (\ref{EulerMaruyamaMV}).

The corresponding probability weights $\{\brh_{k}^{i}\}_{i=1}^{N}$ are updated via Algorithm \ref{ProxRecur}. Notice that computing $\bm{C}_{k}$ requires both $\{\bm{x}_{k-1}^{i}\}_{i=1}^{N}$ and $\{\bm{x}_{k}^{i}\}_{i=1}^{N}$, and that $\bm{C}_{k}$ needs to be passed as input to Algorithm \ref{ProxRecur}. Thus, the execution of Euler-Maruyama scheme precedes that of Algorithm \ref{ProxRecur}.
\begin{figure*}[t]
\centering
\includegraphics[width=.98\linewidth]{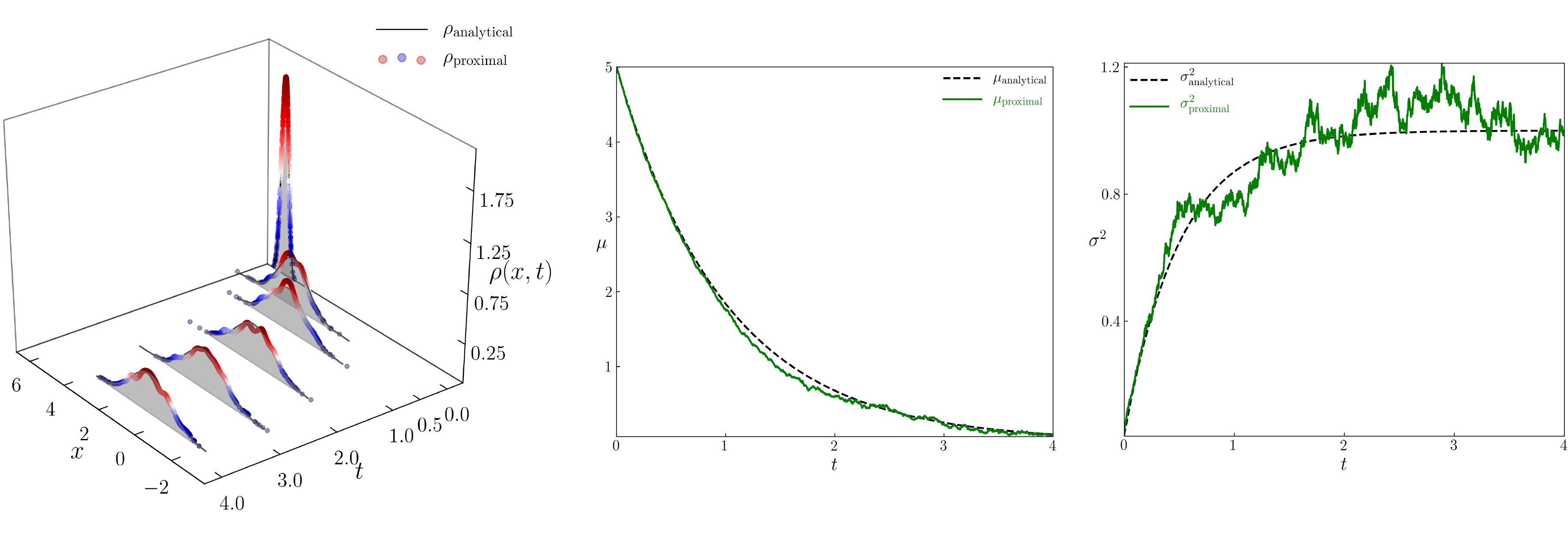}
\vspace*{-0.15in}
\caption{\small{Comparison of the analytical and proximal solutions of the FPK PDE for (\ref{OU}) with time step $h=10^{-3}$, and with parameters $a=1$, $\beta=1$, $\epsilon = 5\times 10^{-2}$. Shown above are the time evolution of the (\emph{left}) PDFs, (\emph{middle}) means, and (\emph{right}) variances.}}
\vspace*{-0.1in}
\label{1dOU}
\end{figure*}

\begin{remark}
Our choice of the (explicit) Euler-Maruyama scheme for updating the location of the points in state space was motivated by its simplicity and ease of implementation.	Since the diffusion coefficient of (\ref{ItoGradient}) or (\ref{MVsde}) is constant, the Euler-Maruyama scheme is guaranteed to converge strongly to the true solution of the corresponding SDE provided the drift coefficient is globally Lipschitz; see e.g., \cite[Ch. 10.2]{kloeden2013numerical}. If the drift coefficient $-\nabla\psi$ in (\ref{ItoGradient}), or $-\left(\nabla \psi + \nabla\left(\rho * \phi\right)\right)$ in (\ref{MVsde}), is non-globally Lipschitz with superlinear growth, then the Euler-Maruyama scheme is known \cite{hutzenthaler2010strong} to diverge in mean-squared error sense. In such cases, one could replace the explicit Euler-Maruyama scheme with the ``tamed" Euler-Maruyama scheme \cite{hutzenthaler2012strong}, or with the partially implicit Euler scheme \cite{hutzenthaler2015numerical}, or with the split-step backward Euler scheme \cite{higham2002strong}.
\end{remark}

\subsection{Convergence} \label{SubsecConvergence}
Next, we will prove the convergence properties for Algorithm \ref{ProxRecur}. To this end, the Definition \ref{DefnThompson} and Proposition \ref{PropOrderPreservingMap} given below will be useful in establishing Theorem \ref{ThmConvergence} that follows. 
 
\begin{definition}(\textbf{Thompson metric})
\label{DefnThompson}
Consider $\bm{z}, \widetilde{\bm{z}} \in \mathcal{K}$, where $\mathcal{K}$ is a non-empty open convex cone. Further, suppose that $\mathcal{K}$ is a normal cone, i.e., there exists constant $\alpha$ such that $\parallel \bm{z} \parallel \leq \alpha \parallel \widetilde{\bm{z}} \parallel$ for $\bm{z} \leq \widetilde{\bm{z}}$. Thompson \cite{thompson1963certain} proved that	$\mathcal{K}$ is a complete metric space w.r.t. the so-called \emph{Thompson metric} given by
\[d_{\rm{T}}\left(\bm{z},\widetilde{\bm{z}}\right) := \max\{\log \gamma(\bm{z}/\widetilde{\bm{z}}), \log \gamma(\widetilde{\bm{z}}/\bm{z})\},\]
where $\gamma(\bm{z}/\widetilde{\bm{z}}) := \inf\{c>0 \mid \bm{z} \leq c\widetilde{\bm{z}}\}$. In particular, if $\mathcal{K}\equiv \mathbb{R}^{N}_{>0}$ (positive orthant of $\mathbb{R}^{N}$), then
\begin{eqnarray}
d_{\rm{T}}\left(\bm{z},\widetilde{\bm{z}}\right) = \log\max\bigg\{\max_{i=1,\hdots,N}\left(\frac{\bm{z}_{i}}{\widetilde{\bm{z}}_{i}}\right), \max_{i=1,\hdots,N}\left(\frac{\widetilde{\bm{z}}_{i}}{\bm{z}_{i}}\right)\bigg\}.
\label{ThompsonPosOrthant}	
\end{eqnarray}
\end{definition}
\begin{proposition}\label{PropOrderPreservingMap}\cite[Proposition 3.2]{lim2009nonlinear},\cite{nussbaum1988hilbert}
Let $\mathcal{K}$ be an open, normal, convex cone, and let $\bm{p} : \mathcal{K} \mapsto \mathcal{K}$ be an order preserving homogeneous map of degree $r \geq 0$, i.e., $\bm{p}(c\bz) = c^{r}\bm{p}(\bm{z})$ for any $c>0$ and $\bm{z}\in\mathcal{K}$. Then, for all $\bm{z}, \widetilde{\bm{z}} \in \mathcal{K}$, we have
\[d_{\rm{T}}\left(\bm{p}(\bm{z}),\bm{p}(\widetilde{\bm{z}})\right) \leq r d_{\rm{T}}\left(\bm{z},\widetilde{\bm{z}}\right).\]
In particular, if $r\in[0,1)$, then the map $\bm{p}(\cdot)$ is strictly contractive in the Thompson metric $d_{\rm{T}}$, and admits unique fixed point in $\mathcal{K}$.
\end{proposition}

Using (\ref{ThompsonPosOrthant}) and Proposition \ref{PropOrderPreservingMap}, we establish the convergence result below.
\begin{thm}\label{ThmConvergence}
Consider the notations in (\ref{yzmaps})-(\ref{Gammaxidef}), and those in Algorithm \ref{ProxRecur}. The iteration
\begin{align}
\bm{z}(:,\ell+1) &= \left(\bm{\xi}_{k-1} \oslash \left(\bm{\Gamma}_{k}^{\top}\bm{y}(:,\ell)\right)\right)^{\frac{1}{1+\beta\epsilon/h}}\nonumber\\
&= \left(\bm{\xi}_{k-1} \oslash \left(\bm{\Gamma}_{k}^{\top}\bm{\brh}_{k-1} \oslash \left( \bm{\Gamma}_{k} \bm{z}(:,\ell)\right)\right)\right)^{\frac{1}{1+\beta\epsilon/h}}
\label{ConvThmIter}	
\end{align}
for $\ell = 1, 2, \hdots$, is strictly contractive in the Thompson metric (\ref{ThompsonPosOrthant}) on $\mathbb{R}^{N}_{>0}$, and admits unique fixed point $\bm{z}^{\rm{opt}} \in \mathbb{R}^{N}_{>0}$. 
\end{thm}
\begin{proof}
Rewriting (\ref{ConvThmIter}) as
\begin{eqnarray*}
\bm{z}(:,\ell+1) = \left(\left(\bm{\xi}_{k-1} \oslash \left(\bm{\Gamma}_{k}^{\top}\bm{\brh}_{k-1}\right)\right) \odot \left( \bm{\Gamma}_{k} \bm{z}(:,\ell)\right)\right)^{\frac{1}{1+\beta\epsilon/h}},	
\end{eqnarray*}
and letting $\bm{\eta} \equiv \bm{\eta}_{k,k+1} := \bm{\xi}_{k-1} \oslash \left(\bm{\Gamma}_{k}^{\top}\bm{\brh}_{k-1}\right)$, we notice that iteration (\ref{ConvThmIter}) can be expressed as a cone preserving composite map $\bm{\theta} := \bm{\theta}_{1} \circ \bm{\theta}_{2} \circ \bm{\theta}_{3}$, where $\bm{\theta} : \mathbb{R}^{N}_{>0} \mapsto \mathbb{R}^{N}_{>0}$, given by 
\begin{eqnarray}
\bm{z}(:,\ell+1) = \bm{\theta}\left(\bm{z}(:,\ell)\right) = \bm{\theta}_{1} \circ \bm{\theta}_{2} \circ \bm{\theta}_{3} \: \left(\bm{z}(:,\ell)\right),
\end{eqnarray}
and $\bm{\theta}_{1}(\bm{z}) := \bm{z}^{\frac{1}{1+\beta\epsilon/h}}$, $\bm{\theta}_{2}(\bm{z}) := \bm{\eta}\odot\bm{z}$, $\bm{\theta}_{3} := \bm{\Gamma}_{k} \bm{z}$. Our strategy is to prove that the composite map $\bm{\theta}$ is contractive on $\mathbb{R}^{N}_{>0}$ w.r.t. the metric $d_{\rm{T}}$.

From (\ref{EDMdef}) and (\ref{Gammaxidef}), $\bm{C}_{k}(i,j)\in[0,\infty)$ which implies $\bm{\Gamma}_{k}(i,j) \in (0,1]$; therefore, $\bm{\Gamma}_{k}$ is a positive linear map for each $k=1,2,\hdots$. Thus, by (linear) Perron-Frobenius theorem, the map $\bm{\theta}_{3}$ is contractive on $\mathbb{R}^{N}_{>0}$ w.r.t. $d_{\rm{T}}$. The map $\bm{\theta}_{2}$ is an isometry by Definition \ref{DefnThompson}. As for the map $\bm{\theta}_{1}$, notice that the quantity $r := 1/(1 + \beta\epsilon/h) \in (0,1)$ since $\beta\epsilon/h > 0$. Therefore, the map $\bm{\theta}_{1}(\bm{z}) := \bm{z}^{r}$ (element-wise exponentiation) is monotone (order preserving) and homogeneous of degree $r\in(0,1)$ on $\mathbb{R}^{N}_{>0}$. By Proposition \ref{PropOrderPreservingMap}, the map $\bm{\theta}_{1}(\bm{z})$ is strictly contractive. Thus, the composition
\[\bm{\theta} = \underbrace{\bm{\theta}_{1}}_{\text{strictly contractive}} \circ \underbrace{\bm{\theta}_{2}}_{\text{isometry}}  \circ \underbrace{\bm{\theta}_{3}}_{\text{contractive}}\]
is strictly contractive w.r.t. $d_{\rm{T}}$, and (by Banach contraction mapping theorem) admits unique fixed point $\bm{z}^{\rm{opt}}$ in $\mathbb{R}^{N}_{>0}$. 
\end{proof}
\begin{corollary}\label{CorrConv}
The Algorithm \ref{ProxRecur} converges to unique fixed point $(\bm{y}^{\rm{opt}},\bm{z}^{\rm{opt}}) \in \mathbb{R}^{N}_{>0} \times \mathbb{R}^{N}_{>0}$.	
\end{corollary}
\begin{proof}
Since $\bm{y}(:,\ell+1) = \bm{\brh}_{k-1} \oslash \left( \bm{\Gamma}_{k} \bm{z}(:,\ell+1)\right)$, the $\bm{z}$ iterates converge to unique fixed point $\bm{z}^{\rm{opt}}\in\mathbb{R}^{N}_{>0}$ (by Theorem \ref{ThmConvergence}), and the linear maps $\bm{\Gamma}_{k}$ are contractive (by Perron-Frebenius theory, as before), consequently the $\bm{y}$ iterates also converge to unique fixed point $\bm{y}^{\rm{opt}}\in\mathbb{R}^{N}_{>0}$. Hence the statement.
\end{proof}

\section{Numerical Results} \label{NumericalResultsSectionLabel}
We now illustrate the computational framework proposed in Section \ref{MainResultSectionLabel} via numerical examples. Our examples involve systems already in JKO canonical form (Section \ref{JKOformsectionlabel}), as well as those which can be transformed to such form by non-obvious change of coordinates.

\subsection{Linear Gaussian Systems}
For an It\^{o} SDE of the form 
\begin{eqnarray}
\differential \bx = \bm{A}\bx \: \differential t + \bm{B} \: \differential \bw, \label{linearSDE}
\end{eqnarray}
it is well known that if $\bx_{0} := \bx(t=0) \sim \mathcal{N}(\bm{\mu}_0,\bm{\Sigma}_0)$,  then the transient joint PDFs $\rho(\bx,t) = \mathcal{N}(\bm{\mu}(t),\bm{\Sigma}(t))$ where the vector-matrix pair $\left(\bm{\mu}(t),\bm{\Sigma}(t)\right)$ evolve according to the ODEs 
\begin{subequations}
\begin{align}
\dot{\bm{\mu}}(t) &= \bm{A}\bm{\mu}, \quad \bm{\mu}(0)=\bm{\mu}_0, \label{meanODE}\\
\dot{\bm{\Sigma}}(t) &= \bm{A}\bm{\Sigma(t)} + \bm{A} \bm{\Sigma(t)}^{\top} + \bm{B} \bm{B}^{\top}, \quad \bm{\Sigma}(0) = \bm{\Sigma}_0.\label{covODE}	
\end{align}
\label{MeanCovODE}	
\end{subequations}
We benchmark the numerical results produced by the proposed proximal algorithm vis-\`{a}-vis the solutions of (\ref{MeanCovODE}). We consider the following two sub-cases of (\ref{linearSDE}).

\subsubsection{Ornstein-Uhlenbeck Process}
We consider the univariate system 
\begin{eqnarray}
\differential x = -ax \: \differential t + \sqrt{2\beta^{-1}} \differential w, \quad a,\beta > 0,	
\label{OU}
\end{eqnarray}
which is in JKO form (\ref{ItoGradient}) with $\psi(x) = \frac{1}{2}ax^{2}$. We generate $N=400$ samples from the initial PDF $\rho_{0} = \mathcal{N}(\mu_{0},\sigma_{0}^{2})$ with $\mu_{0}=5$ and $\sigma_{0}^{2}=4\times 10^{-2}$, and apply the proposed proximal recursion for (\ref{OU}) with time step $h=10^{-3}$, and with parameters $a=1$, $\beta=1$, $\epsilon = 5\times 10^{-2}$. For implementing Algorithm \ref{ProxRecur}, we set the tolerance $\delta = 10^{-3}$, and the maximum number of iterations $L=100$. Fig. \ref{1dOU} shows that the PDF point clouds generated by the proximal recursion match with the analytical PDFs $\mathcal{N}\left(\mu_{0}\exp(-at),(\sigma_{0}^{2}-\frac{1}{a\beta})\exp(-2at) + \frac{1}{a\beta}\right)$, and the mean-variance trajectories (computed from the numerical integration of the weighted scattered point cloud data) match with the corresponding analytical solutions.

\subsubsection{Multivariate LTI}\label{LinGaussLTIsubsubsec}
We next consider the multivariate case (\ref{linearSDE}) where the pair $(\bm{A},\bm{B})$ is assumed to be controllable, and the matrix $\bm{A}$ is Hurwitz (not necessarily symmetric). Under these assumptions, the stationary PDF is $\mathcal{N}(\bm{0},\bm{\Sigma}_{\infty})$ where $\bm{\Sigma}_{\infty}$ is the unique stationary solution of (\ref{covODE}) that is guaranteed to be symmetric positive definite. However, it is not apparent whether (\ref{linearSDE})  can be expressed in the form (\ref{ItoGradient}), since for non-symmetric $\bm{A}$, there does \emph{not} exist constant symmetric positive definite matrix $\bm{\Psi}$ such that $\bm{Ax} = -\nabla\bx^{\top} \bm{\Psi} \bx$, i.e., the drift vector field does not admit a natural potential. Thus, implementing the JKO scheme for (\ref{linearSDE}) is non-trivial in general.

In our recent work \cite{halder2017gradient}, two successive \emph{time-varying} co-ordinate transformations were proposed which can bring (\ref{linearSDE}) in the form (\ref{ItoGradient}), thus making it amenable to the JKO scheme. We apply these change-of-coordinates to (\ref{linearSDE}) with 
\[\bm{A} = \begin{pmatrix}
 -10 & 5\\
 -30 & 0	
 \end{pmatrix}, \quad \bm{B} = \begin{pmatrix}
 2\\2.5	
 \end{pmatrix},
\]
which satisfy the stated assumptions for the pair $(\bm{A},\bm{B})$, and implement the proposed proximal recursion on this transformed co-ordinates with $N=400$ samples generated from the initial PDF $\rho_{0} = \mathcal{N}(\bm{\mu}_{0},\bm{\Sigma}_{0})$, where $\bm{\mu}_{0}=(4,4)^{\top}$ and $\bm{\Sigma}_{0}=4\bm{I}_{2}$. As before, we set $\delta = 10^{-3}, L=100, h=10^{-3}, \beta=1, \epsilon = 5\times 10^{-2}$. Once the proximal updates are done, we transform back the probability weighted scattered point cloud to the original state space co-ordinates via change-of-measure formula associated with the known co-ordinate transforms \cite[Section III.B]{halder2017gradient}. Fig. \ref{Hurwitz} shows the resulting point clouds superimposed with the contour plots for the analytical solutions $\mathcal{N}(\bm{\mu}(t),\bm{\Sigma}(t))$ given by (\ref{MeanCovODE}). Figs. \ref{Hurwitzmean} and \ref{Hurwitzcovar} compare the respective mean and covariance evolution. We point out that the change of co-ordinates in \cite{halder2017gradient} requires implementing the JKO scheme in a time-varying rotating frame (defined via exponential of certain time varying skew-symmetric matrix) that depends on the stationary covariance $\bm{\Sigma}_{\infty}$. As a consequence, the stationary covariance resulting from the proximal recursion oscillates about the true stationary value.

\begin{figure}[h]
\centering
\includegraphics[width=\linewidth]{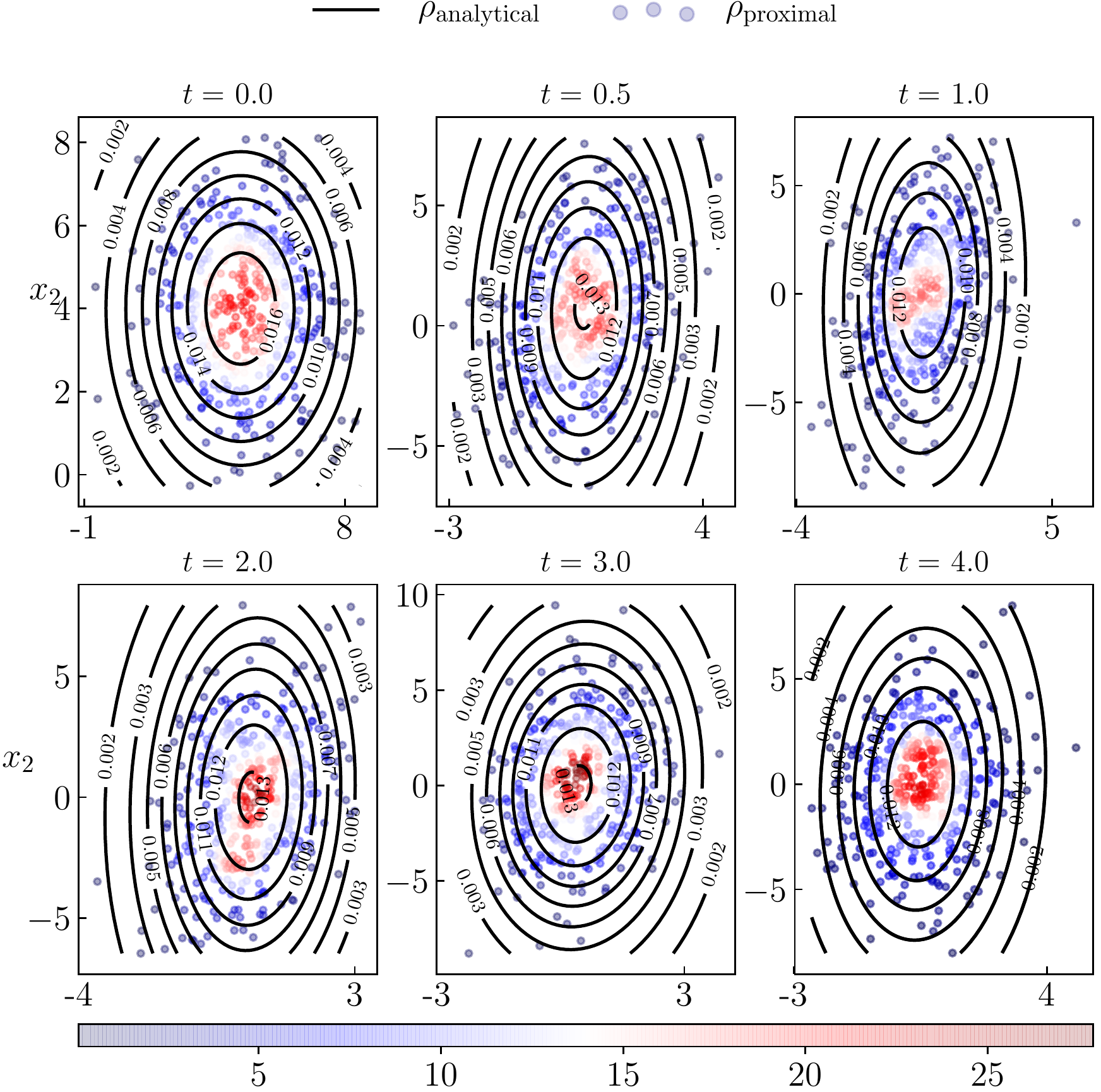}
\vspace*{-0.2in}
\caption{\small{Comparison of the analytical (\emph{contour plots}) and proximal (\emph{weighted scattered point cloud}) joint PDFs of the FPK PDE for (\ref{linearSDE}) with time step $h=10^{-3}$, and with parameters $\beta=1, \epsilon = 5\times 10^{-2}$. Simulation details are given in Section \ref{LinGaussLTIsubsubsec}. The color (\emph{red = high, blue = low}) denotes the joint PDF value obtained via proximal recursion at a point at that time (see colorbar).}}
\vspace*{-0.1in}
\label{Hurwitz}
\end{figure}

\begin{figure}[h]
\centering
\includegraphics[width=.85\linewidth]{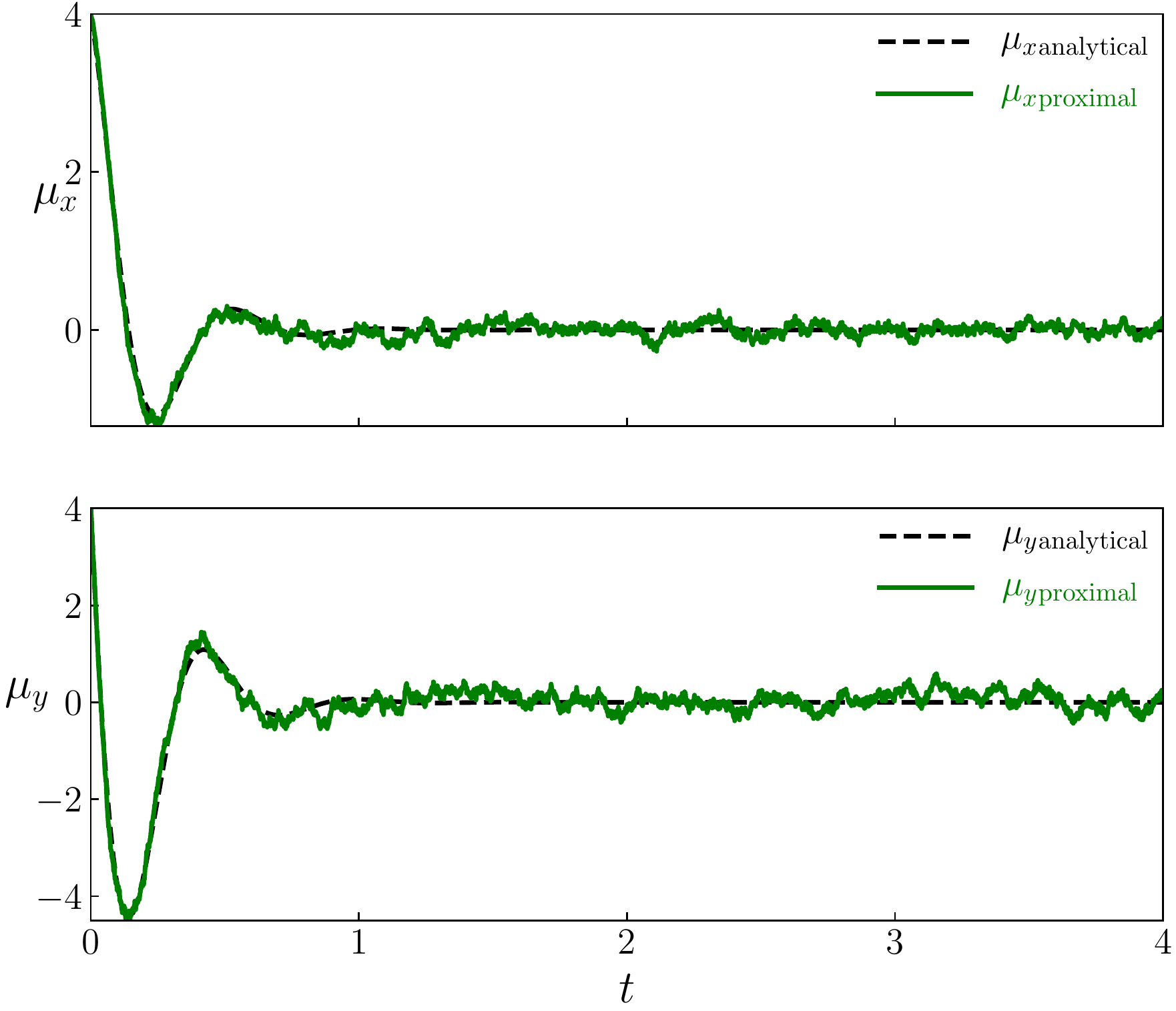}
\vspace*{-0.1in}
\caption{\small{Comparison of the components of the mean vectors from analytical (\emph{dashed}) and proximal (\emph{solid}) computation of the joint PDFs for (\ref{linearSDE}) with time step $h=10^{-3}$, and with parameters $\beta=1, \epsilon = 5\times 10^{-2}$. Simulation details are given in Section \ref{LinGaussLTIsubsubsec}.}}
\vspace*{-0.1in}
\label{Hurwitzmean}
\end{figure}

\begin{figure}[h] 
\centering
\includegraphics[width=.85\linewidth]{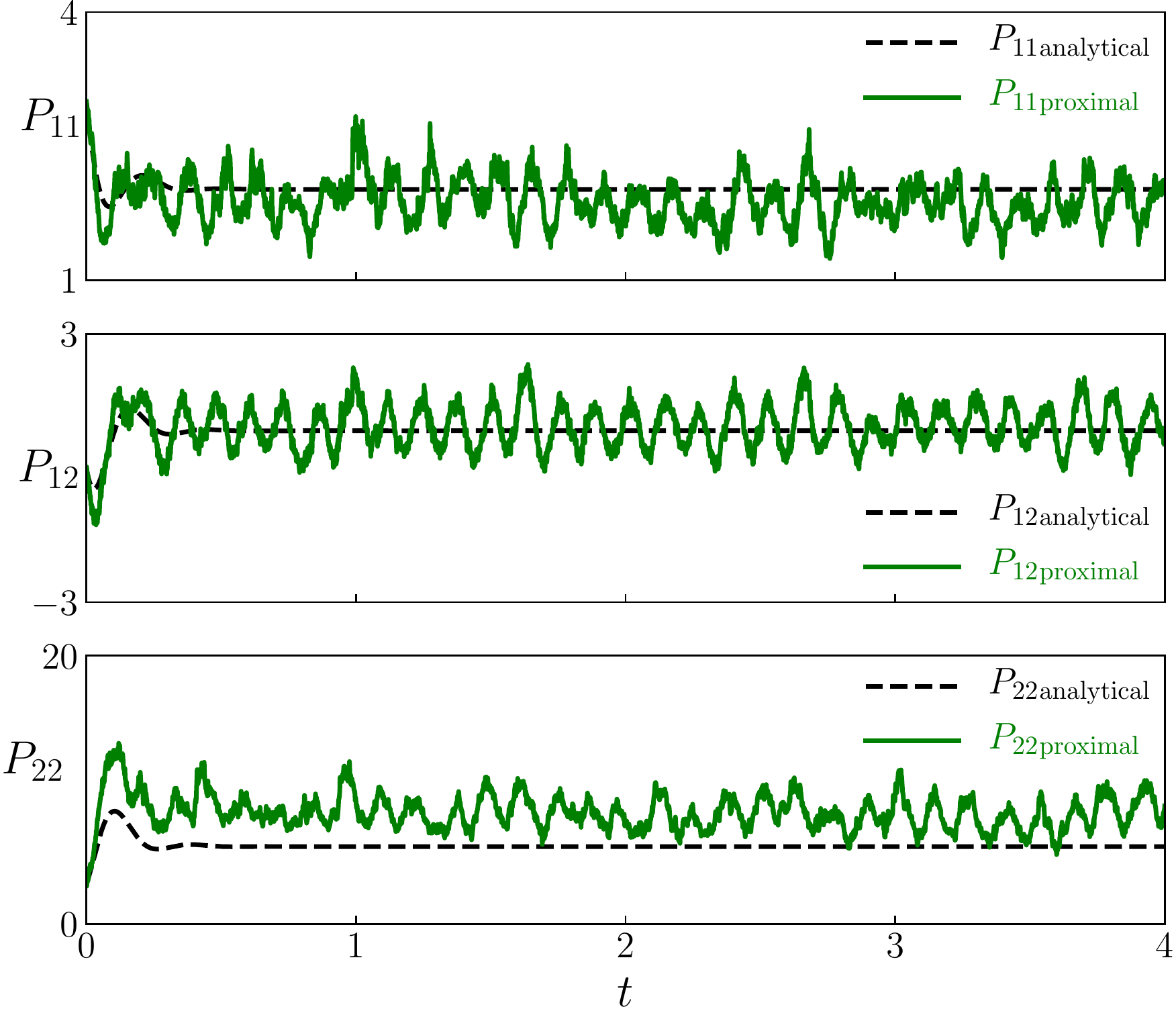}
\vspace*{-0.1in}
\caption{\small{Comparison of the components of the covariance matrices from analytical (\emph{dashed}) and proximal (\emph{solid}) computation of the joint PDFs for (\ref{linearSDE}) with time step $h=10^{-3}$, and with parameters $\beta=1, \epsilon = 5\times 10^{-2}$. Simulation details are given in Section \ref{LinGaussLTIsubsubsec}.}}
\vspace*{-0.1in}
\label{Hurwitzcovar}
\end{figure}

\begin{figure}[tph]
\centering
\includegraphics[width=0.82\linewidth]{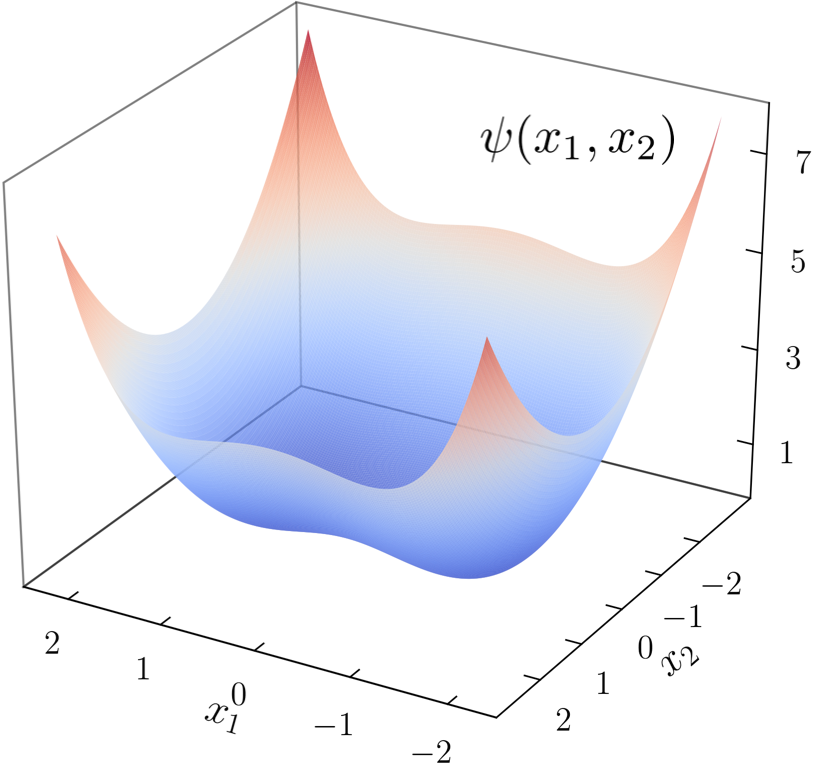}
\vspace*{-0.1in}
\caption{\small{The drift potential $\psi(x_{1},x_{2}) = \dfrac{1}{4}(1+x_1^4) + \dfrac{1}{2}(x_2^2-x_1^2)$ used in the numerical example given in Section \ref{SubsecNonlinNongauss}.}}
\vspace*{-0.1in}
\label{2dpotential}
\end{figure}

\begin{figure}[tph]
\centering
\includegraphics[width=0.95\linewidth]{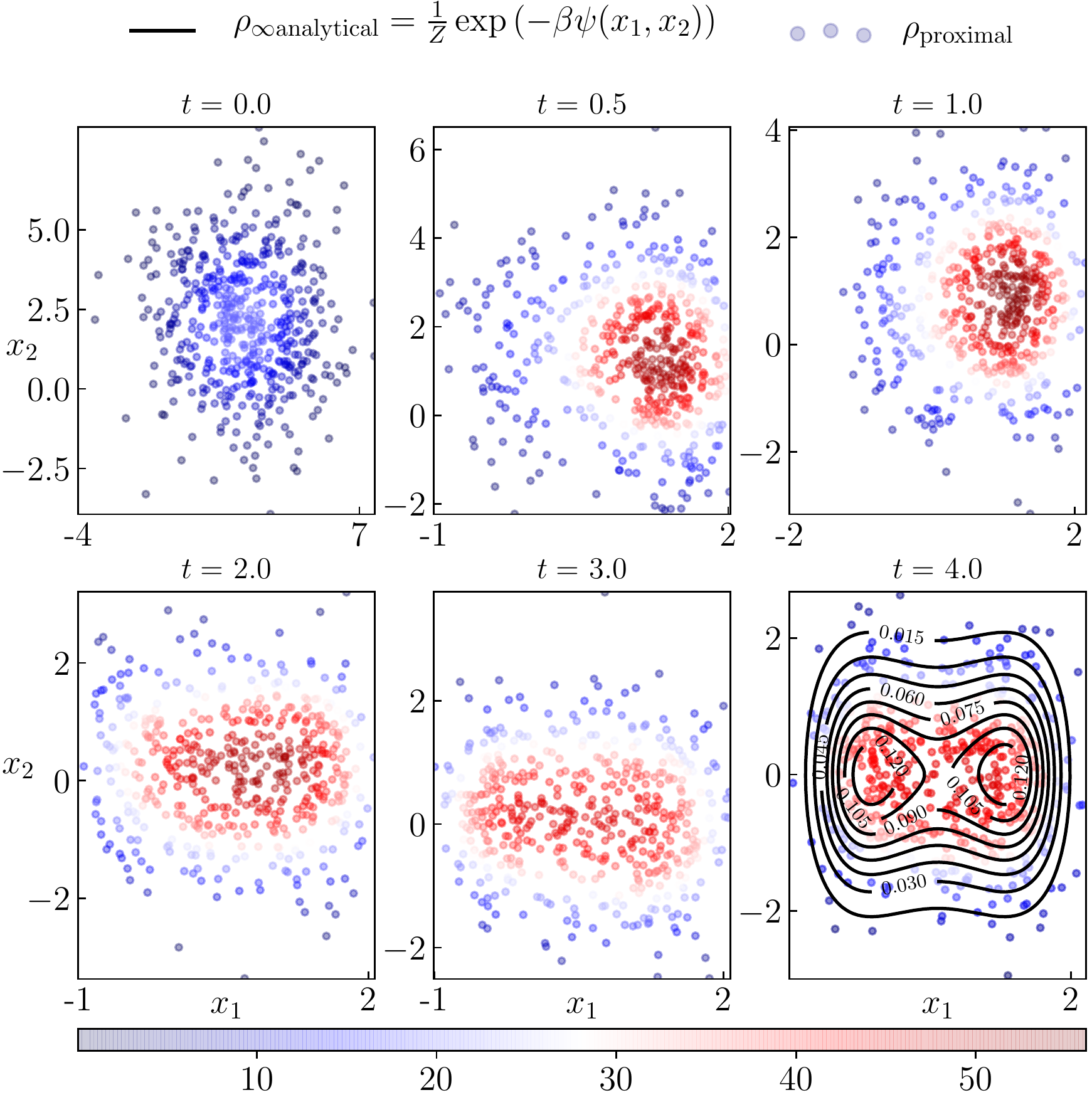}
\vspace*{-0.1in}
\caption{\small{The proximal (\emph{weighted scattered point cloud}) joint PDFs of the FPK PDE (\ref{FPKgradient}) with the drift potential shown in Fig. \ref{2dpotential}, time step $h=10^{-3}$, and with parameters $\beta=1, \epsilon = 5\times 10^{-2}$. Simulation details are given in Section \ref{SubsecNonlinNongauss}. The color (\emph{red = high, blue = low}) denotes the joint PDF value obtained via proximal recursion at a point at that time (see colorbar). In the bottom right plot, the contour lines correspond to the analytical solution for the stationary PDF $\rho_{\infty}$.}}
\vspace*{-0.05in}
\label{2dgrad}
\end{figure}

\subsection{Nonlinear Non-Gaussian System}\label{SubsecNonlinNongauss}
Next we consider a planar nonlinear system of the form (\ref{ItoGradient}) with $\psi(x_{1},x_{2}) = \dfrac{1}{4}(1+x_1^4) + \dfrac{1}{2}(x_2^2-x_1^2)$ (see Fig. \ref{2dpotential}). As mentioned in Section \ref{JKOformsectionlabel}, the stationary PDF is $\rho_{\infty}(\bx) = \kappa \exp\left(-\beta\psi(\bx)\right)$, which for our choice of $\psi$, is bimodal. In this case, the transient PDFs have no known analytical solution but can be computed using the proposed proximal recursion. For doing so, we generate $N=400$ samples from the initial PDF $\rho_{0} = \mathcal{N}(\bm{\mu}_{0},\bm{\Sigma}_{0})$ with $\bm{\mu}_{0}=(2,2)^{\top}$ and $\bm{\Sigma}_{0}=4\bm{I}_{2}$, and set $\delta = 10^{-3}, L=100, h=10^{-3}, \beta=1, \epsilon = 5\times 10^{-2}$, as before. The resulting weighted point clouds are shown in Fig. \ref{2dgrad}; it can be seen that as time progresses, the joint PDFs computed via the proximal recursion, tend to the known stationary solution $\rho_{\infty}$ (contour plots in the right bottom sub-figure in Fig. \ref{2dgrad}).

\begin{figure}[tph]
\centering
\includegraphics[width=0.9\linewidth]{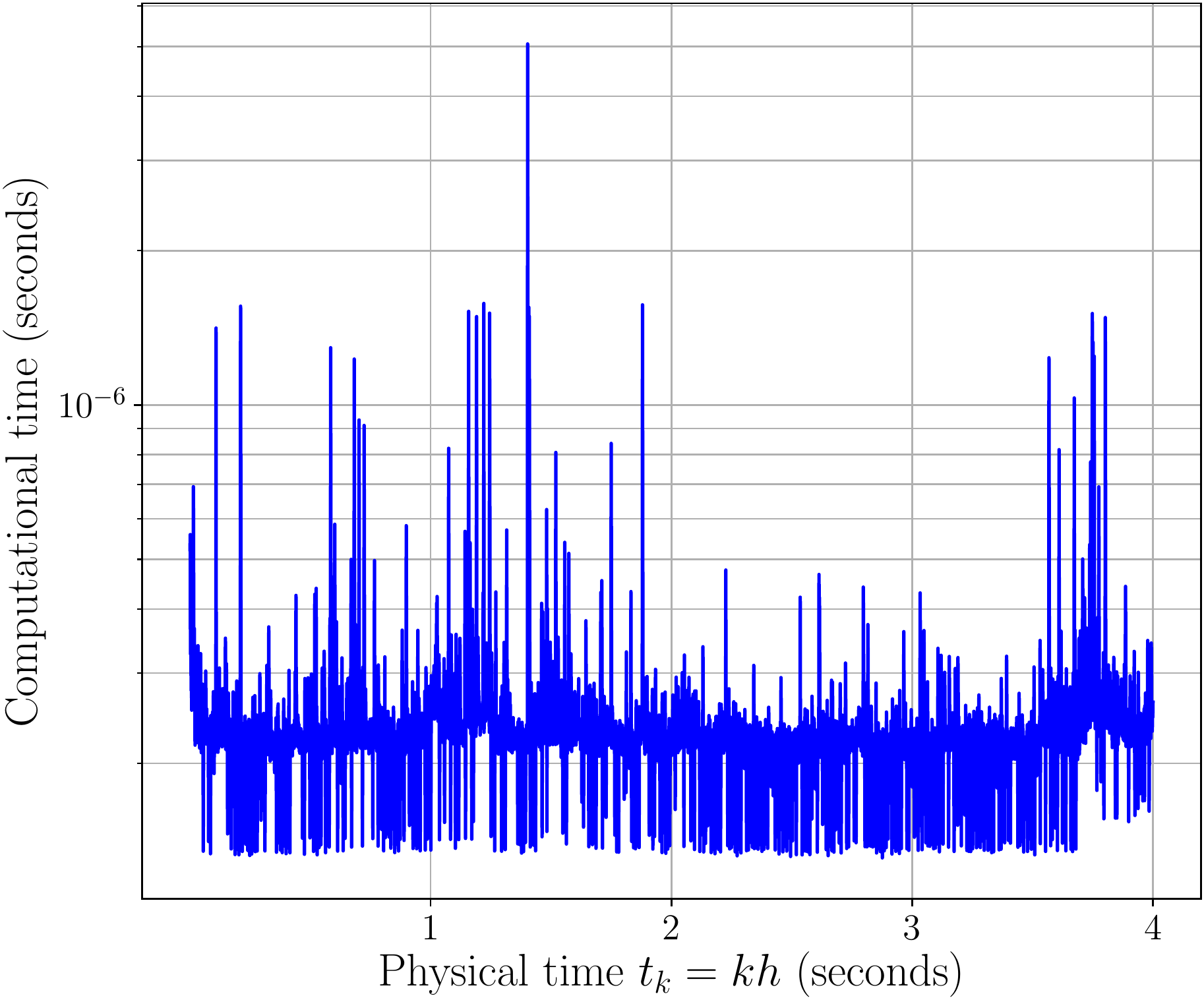}
\caption{\small{The computational times for proximal updates for the simulation in Section \ref{SubsecNonlinNongauss}. Here, the physical time-step $h= 10^{-3}$ s, and $k\in\mathbb{N}$.}}
\vspace*{-0.15in}
\label{RateOfConvProx}
\end{figure}

Fig. \ref{RateOfConvProx} shows the computational times for the proposed proximal recursions applied to the above nonlinear non-Gaussian system. Since the proposed algorithm involves sub-iterations (``while loop" in Algorithm \ref{ProxRecur} over index $\ell\leq L$) while keeping the physical time ``frozen", the convergence reported in Section \ref{SubsecConvergence} must be achieved at ``sub-physical time step" level, i.e., must incur smaller than $h$ (here, $h= 10^{-3}$ s) computational time. Indeed, Fig. \ref{RateOfConvProx} shows that each proximal update takes approx. $10^{-6}$ s, or $10^{-3} h$ computational time, which demonstrates the efficacy of the proposed framework. 

\begin{figure*}[t]
\centering
\includegraphics[width=.98\linewidth]{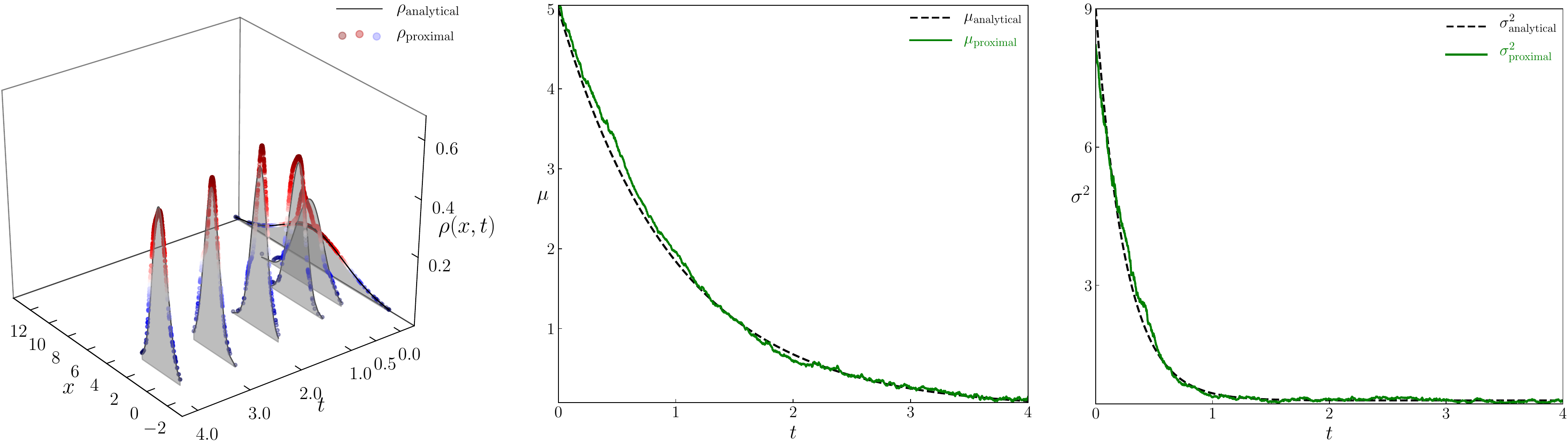}
\caption{\small{Comparison of the analytical and proximal solutions of the McKean-Vlasov flow for (\ref{MeanRevertingSDE}) with time step $h=10^{-3}$, $\rho_{0}=\mathcal{N}(5,9)$, and with parameters $a=b=1$, $\beta=1$, $\epsilon = 5\times 10^{-2}$. Shown above are the time evolution of the transient (\emph{left}) PDFs, (\emph{middle}) means, and (\emph{right}) variances.}}
\vspace*{-0.1in}
\label{MVFPfig}
\end{figure*}

\subsection{Non-local Interactions}\label{SubsecMV}
We now consider a numerical example demonstrating our gradient flow framework for computing the transient PDFs generated by the McKean-Vlasov integro-PDE (\ref{MVFPKgradient}) associated with the density-dependent sample path dynamics (\ref{MVsde}). If both the potentials $\psi$ and $\phi$ are convex functions, then (\ref{MVFPKgradient}) admits unique stationary density $\rho_{\infty}(\bx)$ (see e.g., \cite{carrillo2003kinetic},\cite[Section 5.2]{malrieu2011uniform}). To keep the exposition simple, we consider the univariate case $\psi(x) = \frac{1}{2}ax^{2}$, $\phi(x) = \frac{1}{2}bx^{2}$, $a,b>0$. Performing the integration appearing in the convolution allows us to rewrite (\ref{MVsde}) as the mean-reverting process:
\begin{eqnarray}
\differential x = - \left((a+b)x - b\mu(t)\right)\:\differential t \:+\: \sqrt{2\beta^{-1}}\:\differential w, \quad x(0) = x_{0},
\label{MeanRevertingSDE}	
\end{eqnarray}
where $\mu(t)$ is the mean of the transient PDF $\rho(x,t)$. Assuming $x_{0} \sim \mathcal{N}(\mu_{0}, \sigma_{0}^{2})$, and applying expectation operator to both sides of (\ref{MeanRevertingSDE}) yields $\mu(t) = \mu_{0}\exp(-at)$. Consequently, the transient PDF for (\ref{MeanRevertingSDE}) at time $t$ is $\rho(x,t)=\mathcal{N}\left(\mu(t), \sigma^{2}(t)\right)$, with 
\begin{subequations}
\begin{align}
\mu(t) &= \mu_{0}\exp(-at), \label{meanMV}\\
\sigma^{2}(t) &= \left(\sigma_{0}^{2}-\frac{1}{(a+b)\beta}\right)\exp\left(-2(a+b)t\right) + \frac{1}{(a+b)\beta}.\label{covMV}	
\end{align}
\label{MeanCovMV}	
\end{subequations}
Clearly, the stationary PDF is $\rho_{\infty}(x) = \mathcal{N}\left(0,1/(a+b)\beta\right)$.

To benchmark our algorithm with the analytical solution (\ref{MeanCovMV}), we implement the proximal recursion for (\ref{MeanRevertingSDE}) with free energy (\ref{MVFreeEnergy}). Following \cite[Section 4]{benamou2016augmented}, we replace the non-convex bilinear term $\int_{\mathbb{R}^{n}\times\mathbb{R}^{n}}\phi(\bx - \bm{y})\varrho(\bx)\varrho(\bm{y})\differential\bx\differential\bm{y}$ in (\ref{JKOscheme:b}), with the linear term $\int_{\mathbb{R}^{n}}\phi(\bx - \bm{y})\varrho(\bx)\varrho_{k-1}(\bm{y})\differential\bx\differential\bm{y}$, $k\in\mathbb{N}$,  resulting in a semi-implicit variant of (\ref{JKOscheme}), given by
\begin{eqnarray}
\varrho_{k} =  \underset{\varrho \in \mathscr{D}_2 }\arginf \ \frac{1}{2} W^2(\varrho_{k-1},\varrho) + h\: F(\varrho_{k-1},\varrho), \quad k\in\mathbb{N},
\label{SemiImplicitProxRecursion}	
\end{eqnarray}
with $\varrho_{0} \equiv \rho_{0}(\bx)$ (the initial PDF). The proof for the fact that such a semi-implicit scheme guarantees $\varrho_{k}(\bx) \rightarrow \rho(\bm{x},t=kh)$ for $h\downarrow 0$, where $\rho(\bx,t)$ is the flow generated by (\ref{MVFPKgradient}), can be found in \cite[Section 12.3]{laborde201712}. Notice that for the FPK gradient flow, the ``discrete free energy" in (\ref{FiniteSampleJKO}) was $F(\brh) = \langle \bm{\psi}_{k-1}+\beta^{-1}\log\brh,\brh \rangle$. The scheme (\ref{SemiImplicitProxRecursion}) allows us to write a similar expression in the McKean-Vlasov gradient flow case, as $F(\brh) = \langle \bm{\psi}_{k-1}+\bm{D}_{k-1}\brh_{k-1}+\beta^{-1}\log\brh,\brh \rangle$, where the symmetric matrix $\bm{D}_{k-1}$ is given by 
\[\bm{D}_{k-1}(i,j) = \phi\left(\bx_{k-1}^{i} - \bx_{k-1}^{j}\right), \quad i,j=1,\hdots,N, \quad k \in \mathbb{N}.\]
For the particular choice $\phi(x)=\frac{1}{2}bx^{2}$, notice that $\bm{D}_{k-1}$ is a (scaled) Euclidean distance matrix on the Euler-Maruyama update. In this case, the associated Euler-Maruyama scheme (\ref{EulerMaruyamaMV}) is
\begin{eqnarray}
\bx_{k}^{i} &=& \bx_{k-1}^{i}-h\left(a\bx_{k-1}^{i} + b\left(\bx_{k-1}^{i} - \bm{1}^{\top}\brh_{k-1}\right)\right)\nonumber\\
&+& \sqrt{2\beta^{-1}}\left(\bm{w}_{k}^{i} - \bm{w}_{k-1}^{i}\right), \quad k\in\mathbb{N}, \, i=1,\hdots,N,
\label{EulerMaruyamaMVexample}	
\end{eqnarray}
i.e., the dashed arrow in Fig. \ref{BlockDiagm} becomes active. Fig. \ref{MVFPfig} shows that the weighted scattered point cloud solutions for (\ref{MeanRevertingSDE}) with $a=b=1$, computed through our proximal algorithm match with the analytical solutions $\mathcal{N}(\mu(t),\sigma^{2}(t))$ given by (\ref{MeanCovMV}). For Fig. \ref{MVFPfig}, the parameter values used in our simulation are $h=10^{-3}$, $\beta=1$, $\epsilon = 5\times 10^{-2}$, $\mu_{0}=5$, $\sigma_{0}^{2}=9$, $N=400$, $\delta = 10^{-3}$, $L=100$.

\section{Extensions} \label{ExtensionsSectionLabel}
In Section \ref{LinGaussLTIsubsubsec}, we have already seen that systems not in JKO canonical form may be transformed to the same via suitable change-of-coordinates, thus making density propagation for such systems amenable via our framework. In this Section, we provide two extensions along these lines. First, we consider a case of state-dependent diffusion; thereafter, we consider a system with mixed conservative-dissipative drift. In both cases, we use specific examples (instead of general remarks) to help illustrate the extensions of the basic framework. These examples point out the broad scope of the algorithms proposed herein.

\subsection{Multiplicative Noise}
We consider the It\^{o} SDE for the Cox-Ingersoll-Ross (CIR) model \cite{cox1985theory}, given by 
\begin{eqnarray}
\differential x = a(\theta - x)\:\differential t \: + \: b\sqrt{x}\:\differential w, \quad 2a > b^{2}>0, \quad \theta>0.
\label{CIRsde}
\end{eqnarray}
Due to multiplicative noise, (\ref{CIRsde}) is not in JKO canonical form (\ref{ItoGradient}). If the initial PDF $\rho_{0}$ is Dirac delta at $x_{0}$, then the FPK PDE for (\ref{CIRsde}) admits closed form solution:
\begin{eqnarray}
\rho(x,t) = \begin{cases}c \exp(-(u+v)) \left(\displaystyle\frac{v}{u}\right)^{q/2} I_{q}\left(2\sqrt{uv}\right), &x > 0,\\
0, &\text{otherwise}, \end{cases}
\label{TransientAnalyticalPDF}
\end{eqnarray}
where $I_{q}(\cdot)$ is the modified Bessel function of order $q$, and 
\begin{subequations}
\begin{align}
q &:= \frac{2a\theta}{b^{2}} - 1, &c&:= \displaystyle\frac{2a}{b^{2}\left(1 - \exp(-at)\right)}, \label{qcdef}\\
u &:=  c x_{0} \exp(-at), &v&:= cx. \label{uvdef}
\end{align} 
\label{paramCIRpdf}
\end{subequations}
\noindent The transient solutions (\ref{TransientAnalyticalPDF}) are non-central chi-squared PDFs. The stationary solution is a Gamma PDF $\rho_{\infty}(x) \propto x^{q} \exp\left(-2ax/b^{2}\right)$, $x>0$. We will benchmark our proximal algorithm against (\ref{TransientAnalyticalPDF}).

\begin{figure}[t]
\centering
\includegraphics[width=.9\linewidth]{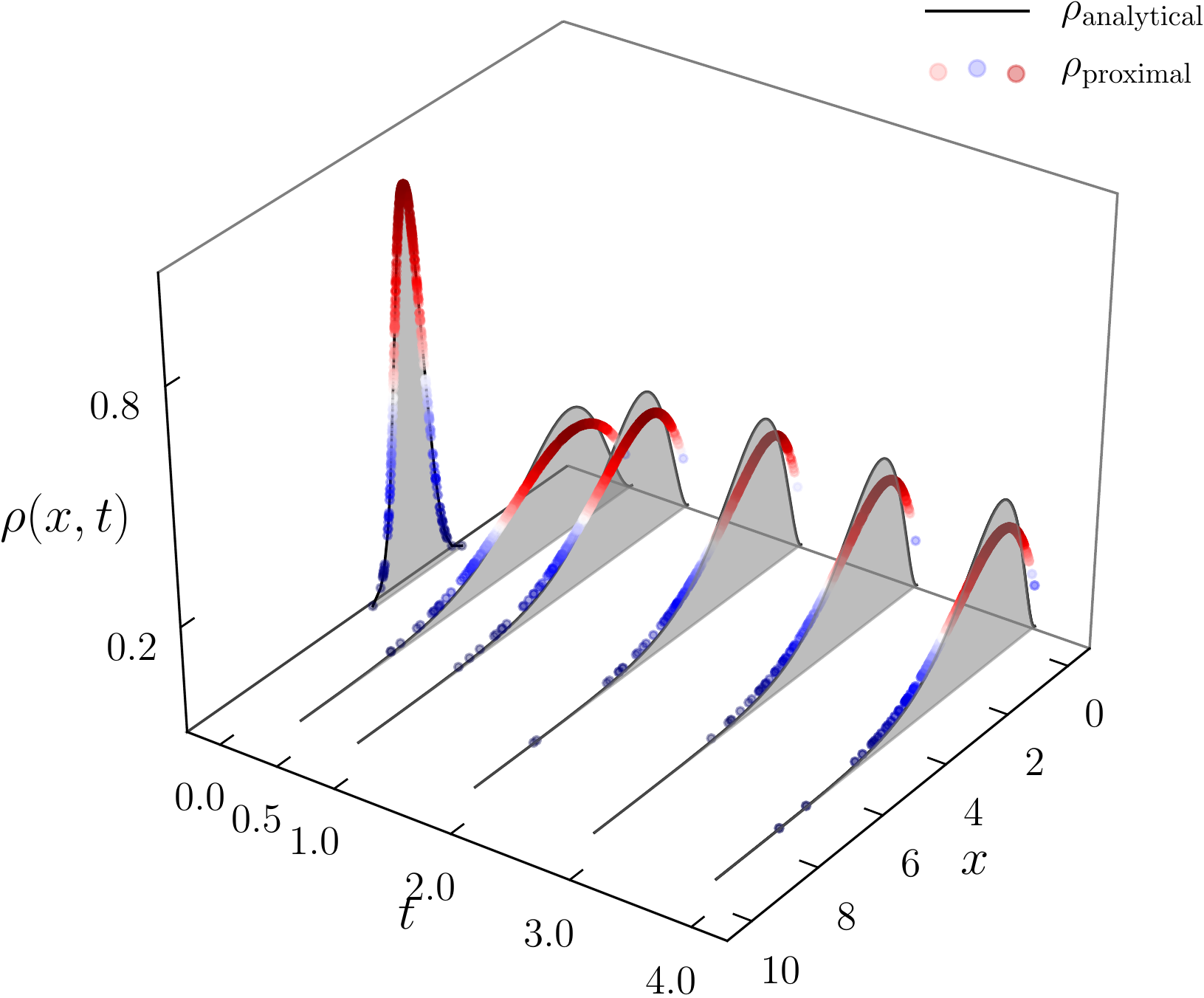}
\caption{\small{Comparison of the analytical and proximal transient PDFs of the FPK PDE for (\ref{CIRsde}) with time step $h=10^{-3}$, and with parameters $a=3$, $b=2$, $\theta = 2$, $x_{0}=5$, $\epsilon = 5\times 10^{-2}$. To approximate the analytical PDFs resulting from $\rho_{0}(x) = \delta(x-5)$, the proximal recursions were performed with the initial PDF $\mathcal{N}(5,10^{-4})$.}}
\vspace*{-0.1in}
\label{CIRfig}
\end{figure}

In order to transcribe (\ref{CIRsde}) in the JKO canonical form, we employ the Lamperti transform \cite{moller2010state,luschgy2006functional}, where the idea is to find a change of variable $y = \varsigma(x)$ such that the SDE for new variable $y$ has unity diffusion coefficient. From It\^{o}'s lemma, it follows that the requisite transformation for (\ref{CIRsde}) is $y = \varsigma(x) := \frac{2}{b}\sqrt{x}$, and the resulting SDE in $y$ becomes  
\begin{eqnarray}
\differential y = \bigg\{ \left(\frac{2a\theta}{b^{2}} - \frac{1}{2}\right)\frac{1}{y} - \frac{a}{2}y\bigg\}\:\differential t \: + \: \differential w.
\label{transformedSDE}
\end{eqnarray}
Clearly, (\ref{transformedSDE}) is in the JKO canonical form (\ref{ItoGradient}) with $\beta=2$, and
\begin{eqnarray}
\psi(y) = ay^{2}/4 - \left(q+1/2\right)\log y.
\end{eqnarray}
So the proposed proximal algorithm (Algorithm \ref{ProxRecur}) can be applied to (\ref{transformedSDE}), and at each time $t$, the resulting PDF $\rho_{Y}(y,t)$ can be transformed back to the PDF $\rho_{X}(x,t)$ via the change-of-measure formula for the push-forward map $x = \varsigma^{-1}(y) = by^{2}/4$. Fig. \ref{CIRfig} shows the comparison of the analytical and proximal transient PDFs for (\ref{CIRsde}) resulting from $\rho_{0} = \delta(x-5)$ with $a=3$, $b=2$, $\theta = 2$, $h=10^{-3}$, $\epsilon = 5\times 10^{-2}$, $N=400$, $\delta=10^{-3}$, and $L=100$.

\subsection{Mixed Conservative-Dissipative Drift}\label{subsecmixed}
In many engineering applications, one encounters It\^{o} SDEs where the drift vector fields have both dissipative (gradient) and conservative (Hamiltonian) components. For example, stochastic systems arising from Newton's law in mechanics often have mixed conservative-dissipative structure. Our intent here is to illustrate that such SDEs are amenable to the proposed proximal recursion framework. As an example, we will work out the details for a perturbed two-body problem in celestial mechanics similar to the one treated in \cite{sun2016uncertainty}.

We consider the relative motion of a satellite in geocentric orbit, given by the second order Langevin equation
\begin{align}
\ddot{\bm{q}} = -\displaystyle\frac{\mu\bm{q}}{\parallel\bm{q}\parallel_{2}^{3}} \:-\:\gamma \dot{\bm{q}} \:&+\:\bm{f}_{\text{pert}}(\bm{q}) \nonumber\\
&+\:\sqrt{2\beta^{-1}\gamma}\:\times\:\text{stochastic forcing}, 
\label{LangevinCelestial} 	
\end{align}
where $\bm{q}:=(x,y,z)^{\top}\in\mathbb{R}^{3}$ is the relative position vector for the satellite, $\mu$ is a constant (product of the Gavitational constant and the mass of Earth), $-\gamma\dot{\bm{q}}$ models linear drag\footnote[2]{More generally, for nonlinear drag of the form $-\gamma\nabla_{\dot{\bm{q}}}\widetilde{V}(\dot{\bm{q}})$, the term $-\gamma\bm{p}$ in (\ref{SDEsuccint}) will become $-\gamma\nabla\widetilde{V}(\bm{p})$. This will entail modifying the functional $\widehat{F}$ in (\ref{JKOschemeKramers}) as $ \mathbb{E}_{\rho}\left[\widetilde{V}(\bm{p}) \: + \: \beta^{-1}\log\rho\right]$. The linear drag illustrated here corresponds to the special case $\widetilde{V}(\bm{p})\equiv \parallel\bm{p}\parallel_{2}^{2}/2$.}, $\bm{f}_{\text{pert}}(\bm{q})$ models (deterministic) perturbative force due to the oblateness of Earth, and the stochastic forcing is due to solar radiation pressure, free-molecular aerodynamic forcing etc. Using the shorthands for sines and cosines as $c\tau := \cos\tau$, $s\tau:=\sin\tau$, and recalling the relations among spherical coordinates $(r,\theta,\phi)$ and cartesian coordinates $(x,y,z)$, given by $r = \sqrt{x^{2} + y^{2} + z^{2}}$, $c\phi = x/r$, $s\phi = \sqrt{1 - (c\phi)^{2}}$, $c\theta = z/r$, $s\theta = \sqrt{1 - (c\theta)^{2}}$, we can write $\bm{f}_{\text{pert}}(\bm{q})$ in spherical coordinates as \cite[eqn. (12)-(13)]{sun2016uncertainty}
\begin{eqnarray}\begin{pmatrix} f_{r}\\
f_{\theta}\\
f_{\phi}
\end{pmatrix}_{\text{pert}} \!\!= \begin{pmatrix} \displaystyle\frac{k}{2r^{4}}\left(3 (s\theta)^{2} - 1\right)\\
-\displaystyle\frac{k}{r^{5}} s\theta\:c\theta\\
0
\end{pmatrix}, \: k := 3J_{2}R_{\rm{E}}^{2}\mu = \:\text{constant},
\label{J2Remu}
\end{eqnarray}
and the same in cartesian coordinates as
\begin{eqnarray}
\bm{f}_{\text{pert}}(\bm{q}) = \begin{pmatrix} f_{x}\\
f_{y}\\
f_{z}
\end{pmatrix}_{\text{pert}} =  \begin{pmatrix}s\theta\:c\phi & c\theta\:c\phi & -s\phi\\
s\theta\:s\phi & c\theta\:s\phi & c\phi\\
c\theta & -s\theta & 0
\end{pmatrix} \begin{pmatrix} f_{r}\\
f_{\theta}\\
f_{\phi}
\end{pmatrix}_{\text{pert}}.
\label{mapping}
\end{eqnarray}
In (\ref{J2Remu}), the Earth oblateness coefficient $J_{2} = 1.082\times 10^{-3}$, the radius of Earth $R_{\rm{E}} = 6.3781\times 10^{6}$ m, and the Earth standard Gravitational parameter $\mu=3.9859\times 10^{14}$ m$^3$/s$^2$. Modeling the stochastic forcing in (\ref{LangevinCelestial}) as standard Gaussian white noise (as in \cite{sun2016uncertainty}), (\ref{LangevinCelestial}) can then be expressed as an It\^{o} SDE in $\mathbb{R}^{6}$:
\begin{eqnarray}
\begin{pmatrix}
\differential x\\
\differential y\\
\differential z\\
\differential v_x\\
\differential v_y\\
\differential v_z
\end{pmatrix} = \begin{pmatrix}
v_x\\
v_y\\
v_z\\
\\
-\displaystyle\frac{\mu x}{r^{3}} + (f_{x})_{\text{pert}} - \gamma v_{x}\\
\\
-\displaystyle\frac{\mu y}{r^{3}} + (f_{y})_{\text{pert}} - \gamma v_{y}\\
\\
-\displaystyle\frac{\mu z}{r^{3}} + (f_{z})_{\text{pert}} - \gamma v_{z}
\end{pmatrix}\differential t + \sqrt{2\beta^{-1}\gamma}\begin{pmatrix}
 	0\\0\\0\\
 	\differential w_{1}\\
 	\differential w_{2}\\
 	\differential w_{3}
 \end{pmatrix},
\end{eqnarray}
or more succinctly,
\begin{eqnarray}
\begin{pmatrix}
 	\differential\bm{q}\\
 	\differential\bm{p}
 \end{pmatrix}
 = \begin{pmatrix}
 	\bm{p}\\
 	-\nabla V(\bm{q}) - \gamma\bm{p}
 \end{pmatrix}\differential t \: + \: \sqrt{2\beta^{-1}\gamma}\:\begin{pmatrix}
\bm{0}_{3\times 1}\\
\differential\bm{w}_{3\times 1}	
\end{pmatrix},
\label{SDEsuccint}	
\end{eqnarray}
where $\bm{p} := (v_{x}, v_{y}, v_{z})^{\top}\in\mathbb{R}^{3}$ is the velocity vector, and 
\begin{subequations}
\begin{align}
&V(\bm{q}) := V_{\text{gravitational}}(\bm{q}) + V_{\text{pert}}(\bm{q}),\label{totalV}\\
&\nabla_{\bm{q}} V_{\text{gravitational}}(\bm{q}) = \displaystyle\frac{\mu\bm{q}}{\parallel\bm{q}\parallel_{2}^{3}}, \: -\nabla_{\bm{q}}V_{\text{pert}}(\bm{q}) = \begin{pmatrix} f_{x}\\
f_{y}\\
f_{z}
\end{pmatrix}_{\text{pert}}.
\end{align}
\label{Vdefn}
\end{subequations} 

Introducing a ``Hamiltonian-like function" $H(\bm{q},\bm{p}) :=  \parallel \bm{p} \parallel_{2}^{2}/2 + V(\bm{q})$, one can verify that the stationary PDF for (\ref{SDEsuccint}) is $\rho_{\infty} \propto \exp(-\beta H(\bm{q},\bm{p}))$, and that the ``total free energy" $F(\rho) := \mathbb{E}_{\rho}\left[H + \beta^{-1}\log\rho\right]$ serves as Lyapunov functional for the associated FPK PDE, i.e., $\frac{\differential F}{\differential t} < 0$. However, the proximal recursion (\ref{JKOscheme:b}) does not apply as is, instead needs to be modified to account the joint conservative-dissipative effect as
\begin{eqnarray}
\varrho_{k} = \underset{\varrho \in \mathscr{D}_2 }\arginf \ \frac{1}{2} \widehat{W}_{h}^2(\varrho_{k-1},\varrho) + h\gamma\: \widehat{F}(\varrho), \quad k\in\mathbb{N}.
\label{JKOschemeKramers}	
\end{eqnarray}
We implement a recursion from \cite[Scheme 2b]{duong2014conservative} where $\widehat{F}(\varrho) := \mathbb{E}_{\rho}\left[\frac{1}{2}\parallel \bm{p}\parallel_{2}^{2} \: + \: \beta^{-1}\log\rho\right]$, and $\widehat{W}_{h}^2(\varrho_{k-1},\varrho)$ is the optimal mass transport cost (as in squared 2-Wasserstein metric (\ref{Wdefn})) with modified cost function (modified integrand in (\ref{Wdefn})), i.e.,
\begin{equation}
	\widehat{W}_{h}^{2}(\rho_{1},\rho_{2}):= \underset{\differential\pi\in\Pi\left(\pi_{1},\pi_{2}\right)}{\inf}\displaystyle\int_{\mathcal{X}\times\mathcal{Y}}\!\!\widehat{s}_{h}\left(\bm{q},\bm{p};\widetilde{\bm{q}},\widetilde{\bm{p}}\right)\:\differential\pi\left(\bm{q},\bm{p},\widetilde{\bm{q}},\widetilde{\bm{p}}\right),\\
\label{ModifiedWass}
\end{equation}
where $(\bm{q},\bm{p})^{\top}$ and $(\widetilde{\bm{q}},\widetilde{\bm{p}})^{\top}$ are two realizations of the state vector governed by (\ref{SDEsuccint})-(\ref{Vdefn}); the integrand in (\ref{ModifiedWass}) is
\begin{align}
\widehat{s}_{h}\left(\bm{q},\bm{p};\widetilde{\bm{q}},\widetilde{\bm{p}}\right) \::=\: &\parallel \widetilde{\bm{p}} - \bm{p} + h\nabla V(\bm{q}) \parallel_{2}^{2} \nonumber\\ 
&+ 12\left\lVert \frac{\widetilde{\bm{q}} - \bm{q}}{h} - \frac{\widetilde{\bm{p}} + \bm{p}}{2} \right\rVert_{2}^{2}.
\label{ModifiedIntegrand}	
\end{align}
That the proximal recursion (\ref{JKOschemeKramers}) with the above choices of functionals $\widehat{F}$ and $\widehat{W}_{h}$ guarantees $\varrho_{k}(\bm{q},\bm{p},h) \rightarrow \rho(\bm{q},\bm{p},t=kh)$ for $k\in\mathbb{N}$ as $h\downarrow 0$, where $\rho(\bm{q},\bm{p},t)$ is the joint PDF generated by the FPK flow for (\ref{SDEsuccint}) at time $t$, was proved in \cite{duong2014conservative}. Our proximal algorithm in Section \ref{SubsecAlgorithm} applies by simply modifying (\ref{Discretepsidef})-(\ref{EDMdef}) as
\begin{eqnarray}
\bm{\psi}_{k-1}(i) &:=& \frac{1}{2}(\bm{p}_{k-1}^{i})^{\top}\bm{p}_{k-1}^{i}, \quad i=1,\hdots,N, \label{modifiedpsidiscrete}\\
\bm{C}_{k}(i,j) &:=& \widehat{s}_{h}(\bm{q}_{k}^{i},\bm{p}_{k}^{i};\bm{q}_{k-1}^{j},\bm{p}_{k-1}^{j}), \quad i,j=1,\hdots,N, \label{modifiedEDMdiscrete}
\label{modifiedpsiEDMdiscrete}	
\end{eqnarray}
respectively. Notice that $\widehat{s}_{h}$ in (\ref{ModifiedIntegrand}) (and consequently, $\widehat{W}_{h}$ in (\ref{ModifiedWass})) is not a metric (in particular, non-symmetric). Thus, the matrices $\bm{C}_{k}$ in (\ref{modifiedEDMdiscrete}) for $k\in\mathbb{N}$ are not symmetric.

To apply the proposed framework, we first non-dimensionalize the variables $\bm{q}$ (m), $\bm{p}$ (m/s), $t$ (s), $\bm{w}$ ($\sqrt{\text{s}}$) in (\ref{SDEsuccint}) as
\begin{eqnarray}
\bm{q}^{\prime} = \bm{q}/R, \quad \bm{p}^{\prime} = \bm{p}/(R/T),\quad t^{\prime} = t/T, \quad \bm{w}^{\prime} = \bm{w}/\sqrt{T},  
\label{NonDimDefn}	
\end{eqnarray}
where $R:=4.2164\times 10^{7}$ m is the radius of the nominal geostationary orbit, and $T=86164$ s is its period. In (\ref{NonDimDefn}), the primed variables are non-dimensionalized. Using (\ref{NonDimDefn}) and It\^{o}'s lemma on (\ref{SDEsuccint}), the non-dimensional SDE in $(\bm{q}^{\prime},\bm{p}^{\prime})^{\top}\in\mathbb{R}^{6}$ becomes
\begin{subequations}
\begin{align}
\differential{\bm{q}^{\prime}} &= \bm{p}^{\prime}\:\differential t^{\prime}, \label{nondim1}\\
\differential{\bm{p}^{\prime}} &= \bigg\{ -\frac{T\mu}{R^{3}}\frac{\bm{q}^{\prime}}{\parallel \bm{q}^{\prime} \parallel_{2}^{3}} + \frac{T^{2}}{R}\bm{f}_{\text{pert}}(R\bm{q}^{\prime}) - \gamma T\bm{p}^{\prime} \bigg\}\differential t^{\prime} \nonumber\\
&\quad + \:\frac{T^{3/2}}{R}\sqrt{2\beta^{-1}\gamma}\:\differential\bm{w}^{\prime}.\label{nondim2}
\end{align} 
\label{SDEnondim}	
\end{subequations}
To avoid numerical conditioning issues, we will apply our algorithm for recursion (\ref{JKOschemeKramers}) associated with the non-dimensional SDE (\ref{SDEnondim}), and then transform the PDF in $(\bm{q}^{\prime},\bm{p}^{\prime})^{\top}$ to the same in original variables $(\bm{q},\bm{p})^{\top}$ via change-of-measure formula. Such a computational pipeline, i.e., density propagation in non-dimensional SDE and then transforming the density back in dimensional variables, is standard in celestial mechanics due to different orders of magnitude in different state variables (see e.g., \cite[Sec. III.B]{sun2016uncertainty}). 

\begin{figure}[hpbt]
\centering
\includegraphics[width=.98\linewidth]{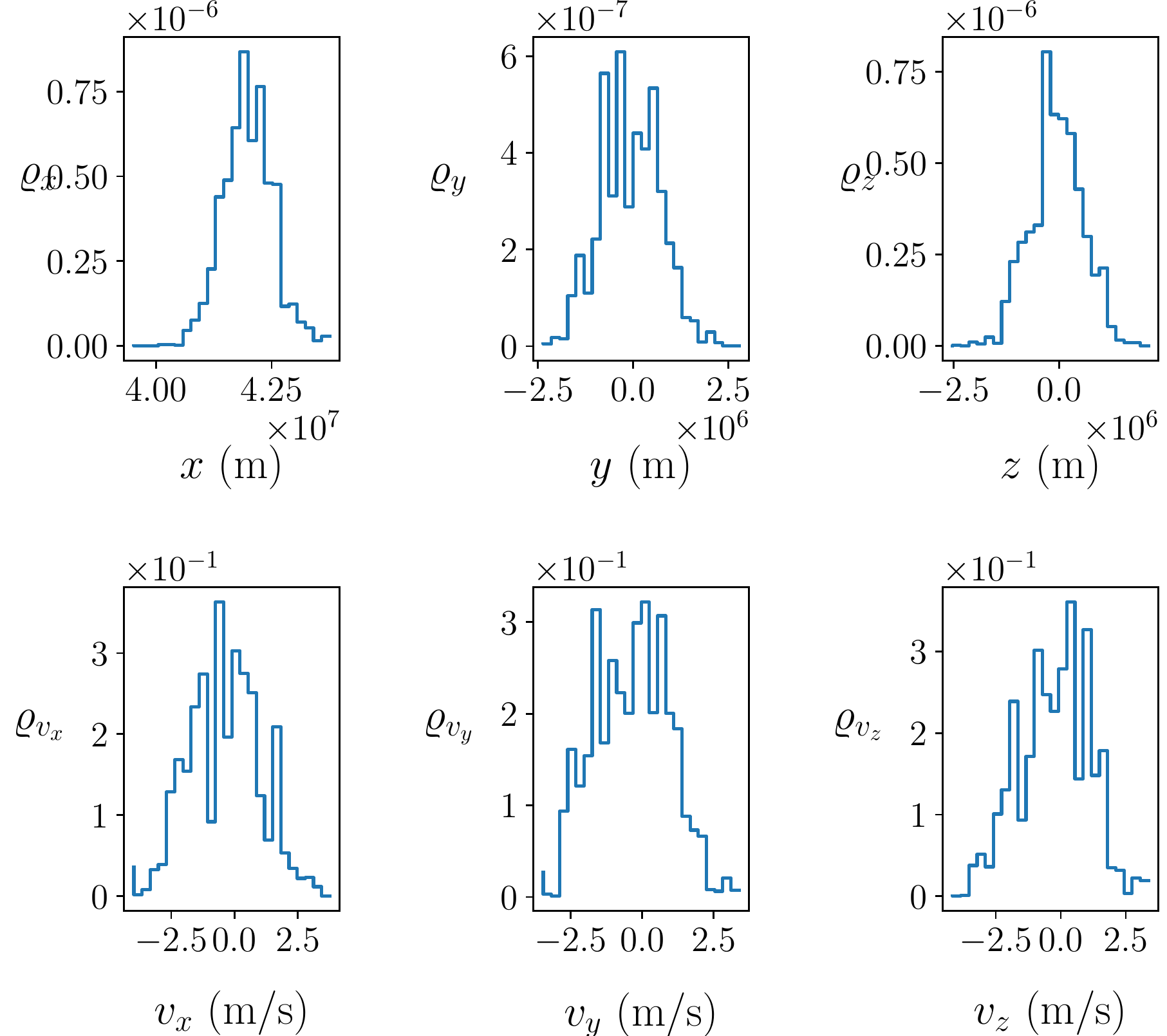}
\caption{\small{Univariate marginal PDFs at $t=0.005$ s for (\ref{SDEsuccint})-(\ref{Vdefn}) computed from the joint PDF at that time obtained via the proposed proximal algorithm for recursion (\ref{JKOschemeKramers}) with time step $h=10^{-5}$, and with parameters $\beta=1$ m$^{2}$/s$^{2}$, $\gamma=1$ s$^{-1}$, $\epsilon = 5\times 10^{-2}$, $\delta=10^{-3}$, $L=100$, and $N=400$.}}
\vspace*{-0.05in}
\label{unimargt1}
\end{figure}

\begin{figure}[hpbt]
\centering
\includegraphics[width=.98\linewidth]{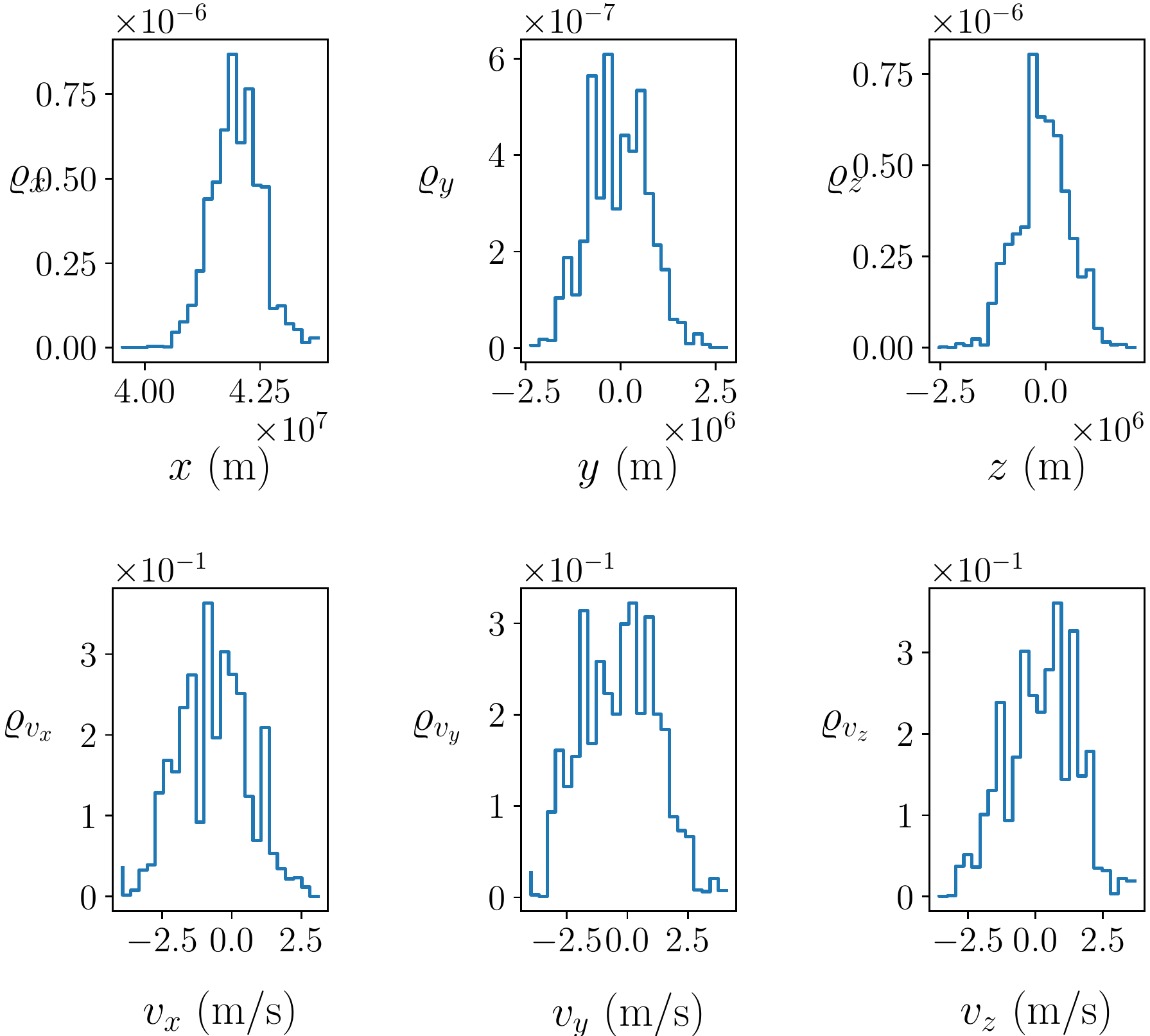}
\caption{\small{Univariate marginal PDFs at $t=0.01$ s for (\ref{SDEsuccint})-(\ref{Vdefn}) computed from the joint PDF at that time obtained via the proposed proximal algorithm for recursion (\ref{JKOschemeKramers}) with time step $h=10^{-5}$, and with parameters $\beta=1$ m$^{2}$/s$^{2}$, $\gamma=1$ s$^{-1}$, $\epsilon = 5\times 10^{-2}$, $\delta=10^{-3}$, $L=100$, and $N=400$.}}
\vspace*{-0.1in}
\label{unimargt4}
\end{figure}

\begin{figure}[tph]
\centering
\includegraphics[width=0.9\linewidth]{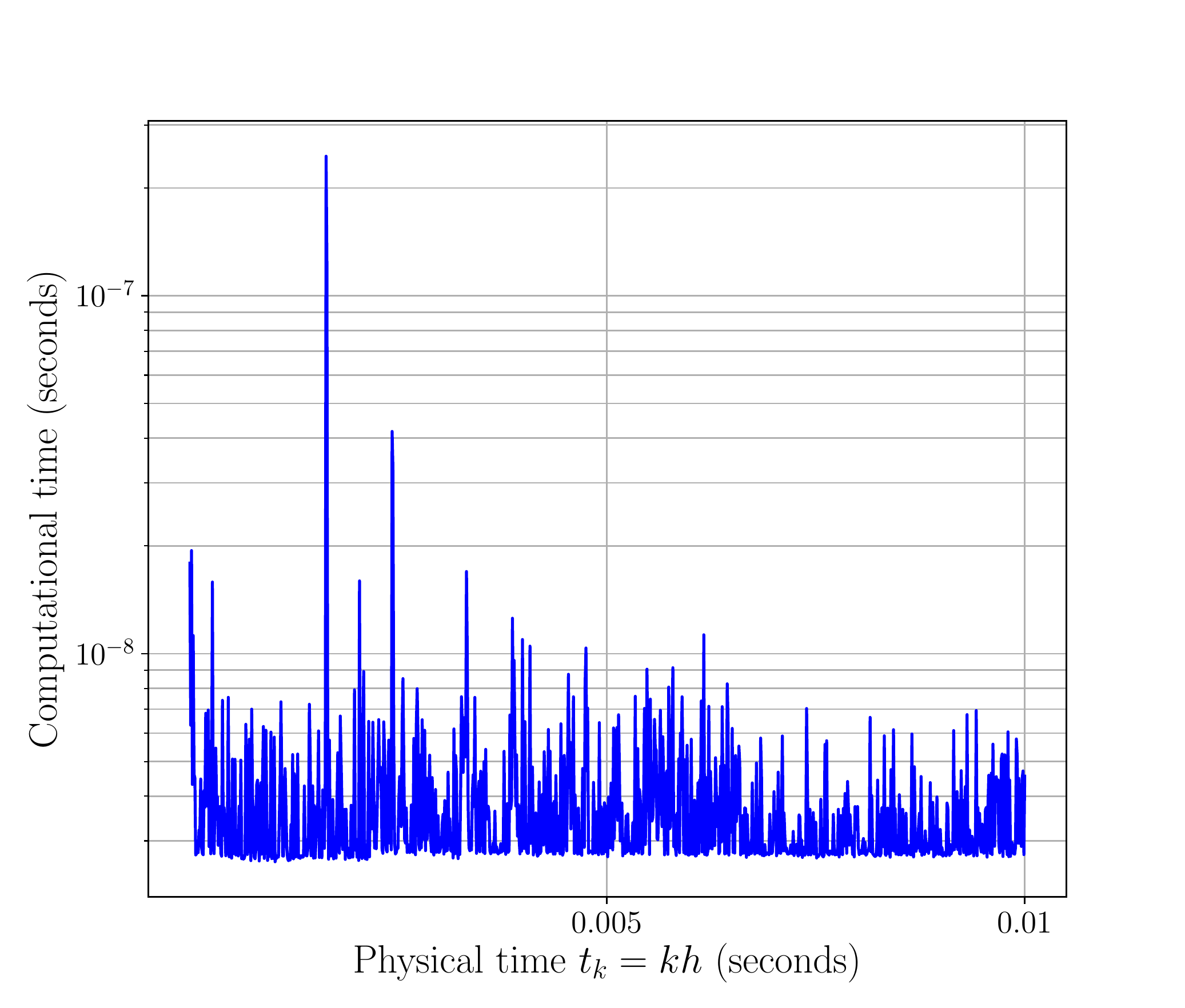}
\caption{\small{The computational times needed for proximal updates in the 6-state numerical example reported in Section \ref{subsecmixed}. Here, the physical time-step $h= 10^{-5}$ s, and $k\in\mathbb{N}$.}}
\vspace*{-0.05in}
\label{CompTimeMixed}
\end{figure}

Figs. \ref{unimargt1} and \ref{unimargt4} show the univariate marginal PDFs at $t=0.005$s and $t=0.010$s respectively, associated with the joint PDFs supported on $\mathbb{R}^{6}$ at those times, computed through the proposed proximal algorithm. For this simulation, the initial joint PDF $\rho_{0}(\bm{q}^{\prime},\bm{p}^{\prime})=\mathcal{N}\left(\bm{\mu}_{0},\bm{\Sigma}_{0}\right)$, where $\bm{\mu}_{0} = \left(1,0,0,0,0,0\right)^{\top}$, and $\bm{\Sigma}_{0}=10^{-4}\times\text{diag}\left(3.335, 6.133, 3.933, 6.562, 9.246, 5.761\right)$. The parameter values used in the simulation are: $h=10^{-5}$, $\beta=1$ m$^{2}$/s$^{2}$, $\gamma=1$ s$^{-1}$, $\epsilon = 5\times 10^{-2}$, $\delta=10^{-3}$, $L=100$, and $N=400$. The computational times for this 6-state example reported in Fig. \ref{CompTimeMixed} reveal that the proximal recursions run remarkably fast (indeed, faster than the runtime for the 2-state example in Section \ref{SubsecNonlinNongauss}, cf. Fig. \ref{RateOfConvProx}).

\section{Conclusions}\label{ConclusionsSectionLabel}
In this paper, novel uncertainty propagation algorithms are presented for computing the flow of the joint PDFs associated with continuous-time stochastic nonlinear systems. By interpreting the PDF flow as gradient descent on the manifold of joint PDFs w.r.t. a suitable metric, the proposed computational framework implements proximal algorithms which are proved to be convergent due to certain contraction properties established herein. Numerical examples are provided to demonstrate the practical use of the proposed algorithm and its efficiency in terms of computational runtime. In contrast to the conventional function approximation algorithms for this problem, the proposed non-parametric framework does not make any spatial discretization, instead performs finite sample probability-weighted scattered data computation in the form of temporal recursion. The location of the atomic measures is delegated to a Euler-Maruyama scheme, thus avoiding spatial discretization. The results of this paper provide computational teeth to the emerging systems-theoretic viewpoint \cite{halder2017gradient,halder2018gradient} that the PDF flows in uncertainty propagation can be seen as gradient flow.

\appendix

\subsection{Derivation of (\ref{SinkhornFormInnerArgmin}):}\label{AppendixDerivation}
\noindent Noting that scaling and translation by constants do not alter the outer argmin in (\ref{EntropyRegJKO}), we rewrite the same as
\begin{align}
\brh_k = h\:\underset{\brh}\argmin \bigg\{  \underset{\bm{M} \in \Pi(\brh_{k-1},\brh)} \min \frac{1}{2h}\langle \bm{C}_{k},\bm{M}\rangle + \frac{\epsilon}{h} H(\bm{M}) \nonumber\\
- \frac{\epsilon}{h}\bm{1}^{\top}\bm{M}\bm{1} + F(\brh)   \bigg\},
\label{EntropyRegJKOscaledANDtranslated}	
\end{align}
since $\bm{1}^{\top}\bm{M}\bm{1}=1$. The Lagrangian $\mathcal{L}$ associated with the inner minimization in (\ref{EntropyRegJKOscaledANDtranslated}) is given by
\begin{align}
\mathcal{L} = \frac{1}{2h}\langle \bm{C}_{k},\bm{M}\rangle + \frac{\epsilon}{h} H(\bm{M}) - \frac{\epsilon}{h}\bm{1}^{\top}\bm{M}\bm{1} + F(\brh) \nonumber\\
+ \langle \bm{\lambda}_{0}, \bm{M}\bm{1} - \brh_{k-1}\rangle + \langle \bm{\lambda}_{1}, \bm{M}^{\top}\bm{1} - \brh\rangle.
\label{InnerLagrangian}	
\end{align}
Setting the derivative of (\ref{InnerLagrangian}) w.r.t. the $(i,j)$-th element of $\bm{M}$ equal to zero, followed by algebraic simplification yields (\ref{SinkhornFormInnerArgmin}).


\balance

\bibliographystyle{IEEEtran}
\bibliography{refs}

%
%

\end{document}